\documentclass[11pt]{article}
\usepackage{amsthm, amsmath, amssymb, amsfonts, booktabs, tikz, setspace, fancyhdr, enumerate}
\usepackage{float}
\usepackage{booktabs}
\usepackage{array}
\usepackage{multirow}
\usepackage{makecell}
\usepackage{xcolor}
\usepackage{caption}
\usepackage{enumitem}
\usepackage[margin = 1in]{geometry}
\usepackage{pgfplots}
\pgfplotsset{compat=1.18}

\usepackage[hidelinks]{hyperref}
\usepackage[capitalise]{cleveref}
\usepackage{tikz}\usetikzlibrary{babel}
\usepackage{comment}
\usepackage{cite}

\usepackage[textsize=scriptsize,colorinlistoftodos]{todonotes}

\newtheorem{theorem}{Theorem}[section]

\newtheorem{lemma}[theorem]{Lemma}
\newtheorem{corollary}[theorem]{Corollary}

\newtheorem{claim}[theorem]{Claim}
\newtheorem{question}[theorem]{Question}

\newtheorem{fact}[theorem]{Fact}

\theoremstyle{definition}
\newtheorem{definition}[theorem]{Definition}

\newtheorem{construction}[theorem]{Construction}

\theoremstyle{remark}

\newenvironment{poc}{\begin{proof}[Proof of claim]}{\end{proof}}

\newcolumntype{C}[1]{>{\centering\arraybackslash}m{#1}} % 定义居中对齐的固定宽度列

\newcommand\restr[2]{{% we make the whole thing an ordinary symbol
  \left.\kern-\nulldelimiterspace % automatically resize the bar with \right
  #1 % the function
  \right|_{#2} % this is the delimiter
  }}

\newcommand{\x}{\times}

\let\leq\le
\let\geq\ge
\let\leqslant\le
\let\geqslant\ge
\begin{document}
\title{The codegree Tur\'an density of tight cycles}
\author{Jie Ma \thanks{School of Mathematical Sciences, University of Science and Technology of China, Hefei, Anhui 230026, and Yau Mathematical Sciences Center, Tsinghua University, Beijing 100084, China.
Research supported by National Key Research and Development Program of China 2023YFA1010201 and National Natural Science Foundation of China grant 12125106. 
Email: \href{mailto:jiema@ustc.edu.cn}{jiema@ustc.edu.cn}}\and  Mingyuan Rong\thanks{School of Mathematical Sciences, University of Science and Technology of China, Hefei,
China.
Research supported by National Key Research and Development Program of China 2023YFA1010201, the NSFC under Grant No. 12125106 and the Excellent PhD Students Overseas Study Program of the University of Science and Technology of China. Email: \href{mailto:rong_ming_yuan@mail.ustc.edu.cn}{rong\_ming\_yuan@mail.ustc.edu.cn}}}
\date{}
\maketitle

\begin{abstract}
The codegree Tur\'an density $\gamma(F)$ of a $k$-uniform hypergraph $F$ is the minimum real number $\gamma \ge 0$ such that every $k$-uniform hypergraph on sufficiently many $n$ vertices, in which every set of $k-1$ vertices is contained in at least $(\gamma+o(1))n$ edges, contains a copy of $F$.
A recent result of Piga, Sanhueza-Matamala, and Schacht determines that $\gamma(C_{\ell}^3)=\frac13$ for every $3$-uniform tight cycle $C_\ell^3$ of length $\ell$, where $\ell \ge \ell_0$ and $\ell$ is not divisible by $3$.
In this paper, we investigate the codegree Tur\'an density of $k$-uniform tight cycles $C_\ell^k$.
We establish improved upper and lower bounds on $\gamma(C_{\ell}^k)$ for general $\ell$ not divisible by $k$.
These results yield the following consequences:
\begin{itemize}
    \item For any prime $k \ge 3$, we show that $\gamma(C_{\ell}^k)=\frac13$ for all sufficiently large $\ell$ not divisible by $k$, generalizing the above theorem of Piga et al.
    \item For all $k \ge 3$, we determine the exact value of $\gamma(C_{\ell}^k)$ for integers $\ell$ not divisible by $k$ in a set of (natural) density at least $\frac{\varphi(k)}{k}$, where $\varphi(\cdot)$ denotes Euler's totient function.  
    \item We give a complete answer to a question of Han, Lo, and Sanhueza-Matamala concerning the tightness of their construction for $\gamma(C_{\ell}^k)$.
\end{itemize}
Moreover, our results also determine the codegree Tur\'an density of $C_\ell^{k-}$, that is, the $k$-uniform tight cycle of length $\ell$ with one edge removed, for a new set of integers $\ell$ of positive density for every $k \ge 3$.
Our upper bound result is based on a structural characterization of $C_{\ell}^k$-free $k$-uniform hypergraphs with high minimum codegree, while the lower bounds are derived from a novel construction model, coupled with the arithmetic properties of the integers $k$ and $\ell$.
\end{abstract}

\section{Introduction}\label{sec 1}

 Let $G$ be a $k$-uniform hypergraph defined by a vertex set $V(G)$ and an edge set $E(G) \subseteq \{S\subseteq V(G): |S|=k\}$. For a fixed $k$-uniform hypergraph $F$, the \emph{Tur\'an number} $\text{ex}(n, F)$ denotes the maximum number of edges in a $k$-uniform hypergraph on $n$ vertices that does not contain a copy of $F$. Determining the value of $\text{ex}(n, F)$ and the \emph{Turán density} $\pi(F) = \lim_{n \to \infty} \text{ex}(n, F)/\binom{n}{k}$
is a central problem in extremal combinatorics. 
While $\pi(F)$ is fully determined for graphs (i.e., $k=2$) by the classic Erd\H{o}s--Stone--Simonovits theorem~\cite{ES1966,ES1946}, the problem is notoriously difficult for $k \ge 3$, where the answer remains unknown even for complete hypergraphs. For more information, we refer to the excellent survey by Keevash~\cite{K2011}.

In this paper, we study a variant of the hypergraph Turán density known as the codegree Turán density, introduced by Mubayi and Zhao~\cite{MZ2007}. For a $k$-uniform hypergraph $G$ and a $(k-1)$-subset $S \subseteq V(G)$, let $N_G(S)$ denote the set $\{v \in V(G) : S \cup \{v\} \in E(G)\}$. We define the \emph{codegree} of $S$ as $|N_G(S)|$, and the \emph{minimum codegree} of $G$, denoted by $\delta_{k-1}(G)$, as the minimum codegree over all $(k-1)$-subsets of $V(G)$. For a $k$-uniform hypergraph $F$, the \emph{codegree Turán number} $\text{ex}_{k-1}(n, F)$ is the largest minimum codegree possible in a hypergraph on $n$ vertices that does not contain a copy of $F$. The \emph{codegree Turán density} of $F$ is defined as:
$$\gamma(F) = \lim_{n \to \infty} \frac{\text{ex}_{k-1}(n, F)}{n}.$$
This limit always exists (see \cite{MZ2007}), and it is clear that $\gamma(F) \le \pi(F)$ holds for any $F$.

Exact values of the codegree Tur\'an density are known for very few hypergraphs.
An early result by Mubayi~\cite{M2005} determined that the codegree Tur\'an density of the Fano plane equals $1/2$, and this was extended to several classes of projective geometries in~\cite{KZ2007,ZG2021}.
A key problem in this study, conjectured by Czygrinow and Nagle~\cite{CN2001}, is whether $\gamma(K_4^3)=1/2$, where $K_4^3$ denotes the complete $3$-graph on $4$ vertices; this remains open.
In contrast, for the $3$-graph $K_4^{3-}$ on $4$ vertices with $3$ edges, Falgas-Ravry, Pikhurko, Vaughan, and Volec~\cite{FPVV2023} established that $\gamma(K_4^{3-}) = 1/4$ using computer-assisted proofs.
Similar methods were applied earlier to show that $\gamma(F_{3,2})=1/3$ in \cite{FMPV2015}, where $F_{3,2}$ denotes the 3-graph $([5],\{123, 124, 125, 345\})$.
Most recently, Ding, Lamaison, Liu, Wang, and Yang~\cite{DLLWY2025} established vanishing codegree density results for a class of layered $3$-graphs.

For integers $k \geq 3$ and $\ell \geq k+1$, a $k$-uniform \emph{tight cycle} $C^k_\ell$ (of length $\ell$) consists of vertices $v_1, \dots, v_\ell$ arranged cyclically, with edges $\{v_i, v_{i+1}, \dots, v_{i+k-1}\}$ for all $i \in [\ell]$, where indices are taken modulo $\ell$. 
Tight cycles play a fundamental role in extremal hypergraph theory, and their density problems has attracted significant attention in recent years. 
A classical result of Erd\H{o}s~\cite{E1964} implies that $\gamma(C^k_\ell)=\pi(C^k_\ell)= 0$ whenever $k \mid \ell$.
For $3$-uniform tight cycles $C^3_\ell$ with $3 \nmid \ell$ and $\ell \ge 7$, the Turán density $\pi(C^3_\ell)$ was determined by Kamčev, Letzter, and Pokrovskiy~\cite{KLP2024} and by Bodnár, León, Liu, and Pikhurko~\cite{BLLP2025}.
Recently, Sankar~\cite{S2024} determined the Tur\'an density for $4$-uniform tight cycles of sufficiently large length.

Turning to codegree Tur\'an densities, 
Han, Lo, and Sanhueza-Matamala~\cite{HLS2021}
showed that $\gamma(C_\ell^k)\leq \frac12$ for any $\ell\geq k+1$, and moreover established the following general lower bound:\footnote{We refer to Subsection~\ref{subsec:1/p} for detailed discussion.} 
\begin{equation}\label{equ:>=1/p}
\gamma(C_\ell^k)\geq \frac{1}{p},
\end{equation}
where $p$ denotes the smallest prime factor of $k/\gcd(k,\ell)$.
They also proved that this lower bound is optimal in the case $p=2$,\footnote{The condition in~\cite{HLS2021} is stated in terms of ``admissible pairs'', which is equivalent to $p=2$ here.} by showing that 
\begin{equation}\label{equ:gamma=1/2}
\gamma(C^k_\ell)=\frac{1}{2}
\mbox{~ for all integers } k\geq 3 \mbox{ and } \ell\geq 2k^2 \mbox{ satisfying } 2\mid \frac{k}{\gcd(k,\ell)}. 
\end{equation}
Motivated by this result, Han, Lo, and Sanhueza-Matamala further asked whether the lower bound \eqref{equ:>=1/p} is tight when $p\geq 3$ and $\gcd(k,\ell)\neq 1$ (see Question 10.2 in \cite{HLS2021}).

The 3-uniform case is particularly interesting.
Note that $C^3_4$ coincides with $K^3_4$.
A result of Balogh, Clemen, and Lidický~\cite{BCL2022} implies $\gamma(C^3_\ell)\leq 0.3993$ for every $\ell\geq 5$, with the exception of $\ell=7$.
Very recently, in a beautiful short paper~\cite{PSS2024}, Piga, Sanhueza-Matamala, and Schacht proved
\[
\gamma(C^3_\ell)=\frac13  \mbox{ ~ for } \ell\in \{10,13,16\} \mbox{ and for every } \ell\geq 19 \mbox{ not divisible by 3}.
\]
See also the follow-up work~\cite{M2024}.
The only remaining open cases are $\gamma(C^3_\ell)$ for $\ell\in \{4,5,7,8\}.$

The main objective of this paper is to investigate the codegree Turán density of $k$-uniform tight cycles $C_\ell^k$. We establish improved upper and lower bounds on $\gamma(C_\ell^k)$ for general $\ell$ not divisible by $k$, which coincide for infinitely many integers $\ell$ for every fixed $k$. In particular, when $k$ is prime or $k\leq 19$, we determine $\gamma(C_\ell^k)$ for all sufficiently large $\ell$.

\subsection{Main results}\label{subsec:Main-results}

Our first main result provides a general upper bound for all cases not determined by~\eqref{equ:gamma=1/2}.

\begin{theorem}
\label{thm:main_upper_bound}
Let $3\leq k< \ell$ be integers with $k \nmid \ell$ and $\ell \ge 20k^2$. If $2\nmid \frac{k}{\gcd(k,\ell)}$, then
$$
\gamma(C^k_\ell) \le \frac{1}{3}.
$$
\end{theorem}

Recall that $p$ denotes the smallest prime factor of $\frac{k}{\gcd(k,\ell)}$.
We observe that the condition $2\nmid \frac{k}{\gcd(k,\ell)}$ is equivalent to $p\ge 3$.
In the case $p=3$, this upper bound matches the lower bound~\eqref{equ:>=1/p} from~\cite{HLS2021}, and thus
extends the results of~\cite{PSS2024, M2024} for $3$-uniform tight cycles to higher uniformities.

The proof of Theorem~\ref{thm:main_upper_bound} relies on the following structural characterization of $C_\ell^k$-free hypergraphs with large minimum codegree, which may be of independent interest.
Note that this characterization immediately implies the conclusion of Theorem~\ref{thm:main_upper_bound} for all odd integers $k$.

\begin{theorem}
\label{thm:structural}
Let $3\le k< \ell$ be integers with $k \nmid \ell$ and $\ell \geqslant 20k^2$. 
Suppose that $G$ is a $k$-uniform hypergraph that contains no homomorphic copy of $C^k_\ell$ and satisfies $\delta_{k-1}(G) > \frac{1}{3} |V (G)|$.
Then there exists a vertex set $B \subseteq V (G)$ with $\frac{|V (G)|}{3} \leqslant |B| \leqslant \frac{2|V (G)|}{3}$ such that for every $e \in E(G)$, both $|e \cap B|$ and $|e\cap (V(G)\setminus B)|$ are odd. 
In particular, $k$ has to be even.
\end{theorem}

Before proceeding, we consider the case \(p>3\).
Unlike the case \(p=3\), the lower bound \(1/p\) in \eqref{equ:>=1/p} is strictly smaller than our upper bound \(1/3\).

Our second main result provides a refined general lower bound, yielding the matching value \(1/3\) for infinitely many pairs \((k,\ell)\) under certain arithmetic conditions, even for \(p>3\).
This bound is obtained via a novel construction, the \((k,\ell;d)\)-family; see Definition~\ref{def:kld_family} for details.

\begin{theorem}\label{thm:new_lower_bound_revised}
Let $3\le k< \ell$ be integers with $k \nmid \ell$.  Let $p$ be the smallest prime factor of $\frac{k}{\gcd(k,\ell)}$, and let $t$ be the smallest odd integer satisfying $t \ge \max(3, \gcd(k,\ell))$. Then
\[
    \gamma(C^k_\ell) \ge \max\left(\frac{1}{p}, \frac{1}{t}\right).
\]
\end{theorem}

The combination of Theorem~\ref{thm:main_upper_bound} and Theorem~\ref{thm:new_lower_bound_revised} yields exact results for a large class of pairs $(k,\ell)$. 
In particular, when $\ell\geq 20k^2$ and $\gcd(k, \ell)=1$, upper and lower bounds coincide as follows:
\begin{equation}\label{equ:gcd=1}
\gamma(C^k_\ell) =
\begin{cases}
1/2 & \text{if } k \text{ is even,} \\
1/3 & \text{if } k \text{ is odd.}
\end{cases}
\end{equation}
This leads to the following immediate corollary. 
Let \(\varphi(\cdot)\) denote Euler’s totient function.

\begin{corollary}\label{cor:prime_and_density}
\begin{itemize}
\item[(a).] For any prime $k\geq 3$, $\gamma(C^k_\ell) = \frac{1}{3}$ holds for all $\ell \ge 20k^2$ with $k \nmid \ell$.
\item[(b).] For any integer \(k\) of the form \(2^\alpha 3^\beta\), \(2q\), or \(3q\) with \(q\) prime, the exact value of \(\gamma(C_\ell^k)\) is determined for all \(\ell \ge 20k^2\) with \(k\nmid \ell\).
\item[(c).] For any fixed $k$, the exact value of $\gamma(C_{\ell}^k)$ is determined for $\ell$ in a set of natural density $\frac{\varphi(k)}{k}$.
\end{itemize}
\end{corollary}
\begin{proof}
Throughout, we assume that \(\ell \ge 20k^2\) and \(k \nmid \ell\).
First, suppose that \(k \ge 3\) is prime. Then \(\gcd(k,\ell)=1\) and \(k\) is odd, so by \eqref{equ:gcd=1} we have \(\gamma(C_\ell^k)=1/3\).
Next, let \(k=2^\alpha 3^\beta\), \(2q\), or \(3q\), 
where we may assume \(q\) is a prime with \(q \ge 5\).
For any such \(\ell\), either \(p\in\{2,3\}\), or \(p=q\ge 5\) with \(2 \le \gcd(k,\ell) \le 3\).
In either case, by \eqref{equ:gamma=1/2} and Theorems~\ref{thm:main_upper_bound} and \ref{thm:new_lower_bound_revised}, we have
\(\gamma(C_\ell^k)\in\{\tfrac12,\tfrac13\}\).
Lastly, for any integer \(k \ge 3\), the equation \eqref{equ:gcd=1} determines the exact value of \(\gamma(C_\ell^k)\) for all \(\ell \ge 20k^2\) with \(\gcd(k,\ell)=1\).
The set of such integers \(\ell\) has natural density \(\varphi(k)/k\).
\end{proof}

In particular, it can be deduced from items (a) and (b) of Corollary~\ref{cor:prime_and_density} that for any $3\leq k\leq 19$, the exact value of $\gamma(C_{\ell}^k)$ is determined for all sufficiently large $\ell$.
The first open case is when $k=20$.

Finally, we obtain a complete answer to the aforementioned question of Han et al.~\cite{HLS2021} on whether the lower bound $\gamma(C_\ell^k)\geq \frac{1}{p}$ is tight for \(p\ge 3\) and \(\gcd(k,\ell)\neq 1\). It is worth noting that the answer is unchanged whether or not the assumption \(\gcd(k,\ell)\neq 1\) is imposed.

\begin{corollary}\label{cor:tightness_rephrased}
Let $3\le k< \ell$ be integers with $k \nmid \ell$ and $\ell \ge 20k^2$. Let $p$ be the smallest prime factor of $\frac{k}{\gcd(k,\ell)}$. 
The tightness of the lower bound $\gamma(C_\ell^k)\geq \frac{1}{p}$ is characterized as follows:
\begin{enumerate}
    \item For $p=2$, the bound was shown to be tight in~\cite{HLS2021}, with $\gamma(C^k_\ell) = \frac12$.
    \item For $p=3$, the bound is tight due to Theorem~\ref{thm:main_upper_bound}, with $\gamma(C^k_\ell) = \frac13$.
    \item For $p>3$, there exists infinitely many tight cycles $C_\ell^k$ for which $\gamma(C_\ell^k)>\frac{1}{p}$.\footnote{For instance, this occurs whenever $k \equiv i \pmod {6i}$ and $\gcd(k,\ell)=i$ for every $i\in \{1,2,3\}$.}
\end{enumerate}
\end{corollary}

\subsection{Implications to tight cycles minus an edge}

The $k$-uniform \emph{tight cycle minus an edge}, denoted by $C^{k-}_\ell$, is the hypergraph obtained from $C^k_\ell$ by removing a single edge. As a natural variant of the tight cycle, this structure has been actively studied. 
The Turán density of the $3$-uniform $C_\ell^{3-}$ was determined for sufficiently large $\ell$ not divisible by $3$ by Balogh and Luo~\cite{BL2024}, 
and was recently completed for all such $\ell \ge 5$ independently by Bodnár, León, Liu, and Pikhurko~\cite{BLLP2025minus} and by Lidický, Mattes, and Pfender~\cite{LMP2024}.
Regarding the codegree Turán density, 
Piga, Sales, and Schülke~\cite{PSS2023} proved for the 3-uniform case that $\gamma(C^{3-}_\ell) = 0$ for all $\ell \ge 5$. 
Recently, Sarkies~\cite{S2025} extended this to general uniformity, showing that 
\[
\gamma(C^{k-}_\ell) = 0 \mbox{ if and only if } \ell \equiv 0, \pm 1 \pmod k.
\]
In the remaining cases, Sarkies~\cite{S2025} showed that $\gamma(C^{k-}_\ell)\geq \gcd(k, \ell)/k$ if $\gcd(k, \ell) > 1$; in the co-prime case ($\gcd(k, \ell) = 1$), 
the lower bound he provided is minuscule, albeit strictly positive. 

We establish stronger lower bounds in both scenarios. 
These new lower bounds are substantial and, in many instances, match the following upper bound  inherited from \cref{thm:main_upper_bound}
\begin{equation}\label{equ:minus-upper}
\gamma(C^{k-}_\ell) \leq \gamma(C^{k}_\ell) \leq \frac{1}{3} ~ \mbox{ if } ~ 2\nmid \frac{k}{\gcd(k,\ell)}.
\end{equation}
We first present the improvement for $\gcd(k, \ell) > 1$, derived via a variant of the proof of Theorem~\ref{thm:new_lower_bound_revised}.

\begin{theorem}\label{thm:minus_edge_main}
Let $3\le k< \ell$ satisfy $\ell \not\equiv 0, \pm 1 \pmod k$ and $\gcd(k,\ell) > 1$.
Let $p$ be the smallest prime factor of $\frac{k}{\gcd(k,\ell)}$ and let $t$ be the smallest odd integer satisfying $t \ge \max(3, \gcd(k,\ell))$. Then
    \[ \gamma(C^{k-}_\ell) \ge \max\left(\frac{1}{p}, \frac{1}{t}\right). \]
\end{theorem}

For the co-prime case $\gcd(k,\ell) = 1$, we establish a lower bound $1/3$ for numerous pairs $(k,\ell)$, 
which, in addition, satisfy certain arithmetic conditions. 

\begin{theorem}\label{thm:minus_edge_main_2}
Let $3\le k< \ell$ be integers with $\ell \not\equiv 0, \pm 1 \pmod k$ and $\gcd(k,\ell) = 1$. If $\ell \not\equiv \pm 2 \pmod k$ and $3\ell \not\equiv \pm 1, \pm 2 \pmod k$, then
    \[ \gamma(C^{k-}_\ell) \ge \frac{1}{3}. \]
\end{theorem}

The proof of this lower bound uses a refined version of the  $(k,\ell;d)$-family and requires additional treatments compared to the previous constructions.  
We include its proof in \hyperref[sec: appendix1]{the Appendix}. 

Combining Theorems~\ref{thm:minus_edge_main} and \ref{thm:minus_edge_main_2} with the upper bounds \eqref{equ:gamma=1/2} and \eqref{equ:minus-upper}, we are able to determine the exact value of $\gamma(C_\ell^{k-})$ for infinitely many pairs $(k,\ell)$,
including a new set of integers $\ell$ of positive natural density for every fixed $k$.
We refer to Subsection~\ref{subsec:summary} for a summary of results, with details.

\subsection{Summary of Results}\label{subsec:summary}
We summarize the current best known bounds for $\gamma(C^k_\ell)$ and $\gamma(C^{k-}_\ell)$ in \cref{tab:summary_optimized}.

\begin{table}[htbp]
\centering
\caption{Best-known bounds for $\gamma(C^k_\ell)$ and $\gamma(C^{k-}_\ell)$ for sufficiently large $\ell$.}
\label{tab:summary_optimized}
\begin{tabular}{@{} 
    >{\raggedright\arraybackslash}p{3.3cm}
    >{\raggedright\arraybackslash}p{4.5cm}
    >{\raggedleft\arraybackslash}p{1.5cm}
    >{\centering\arraybackslash}p{4.2cm} 
    >{\centering\arraybackslash}p{1.8cm}
@{}}
\toprule
\textbf{Case Conditions} & \multicolumn{2}{l}{\textbf{Specific Parameters}} & \textbf{Bounds for $\gamma$} & \textbf{Ref.}\\
\midrule[\heavyrulewidth]

% --- Degenerate Cases ---
\multirow{2}{3.4cm}[-0.3em]{\textbf{Degenerate Cases}}
& \multicolumn{2}{l}{For $C^k_\ell$, when $k \mid \ell$} & $\gamma = 0$ & \cite{E1964} \\
\cmidrule(l){2-5}
& \multicolumn{2}{l}{For $C^{k-}_\ell$, when $\ell \equiv 0, \pm 1 \pmod k$} & $\gamma = 0$ & \cite{S2025,PSS2023}\\
\midrule[\heavyrulewidth]

% --- Middle Block ---
\multirow{4}{3.4cm}[-0.4em]{$\bullet$ For $C^k_\ell$ where $k \nmid \ell$ \\[0.3em]  $\bullet$ For $C^{k-}_\ell$ where $\ell \not\equiv 0,\pm 1 \pmod k$ and $\gcd(k,\ell)>1$}
& \multicolumn{2}{l}{$p=2$} & $\gamma = 1/2$ & \cite{HLS2021} \\
\cmidrule(l){2-5}
& \multicolumn{2}{l}{$p=3$} & $\gamma = 1/3$ & $\hspace{0.24cm}\bigstar\hspace{0.13cm}$ \\
\cmidrule(l){2-5}
& \multicolumn{2}{l}{$p\geq 4$ and $\gcd(k,\ell) \le 3$} & $\gamma = 1/3$ & $\hspace{0.24cm}\bigstar\hspace{0.13cm}$ \\
\cmidrule(l){2-5}
& \multicolumn{2}{l}{$p\geq 4$ and $\gcd(k,\ell) \geq 4$} & $\max(1/p, 1/t) \le \gamma \le 1/3$ & $\hspace{0.24cm}\bigstar\hspace{0.13cm}$\,/\!$\hspace{0.24cm}\bigstar\hspace{0.13cm}$ \\
\midrule[\heavyrulewidth]

% --- Bottom Block: Complex Cases ---

% Sub-case 1
& \multirow{2}{=}[-0.3em]{%
    \raggedright
    If $\ell \equiv \pm 2 \pmod k$, \\
   or $3\ell \equiv \pm 1, \pm 2 \pmod k$%
}
&  $p=2$  & $0 < \gamma \le 1/2$ & \cite{S2025}\,/\!\cite{HLS2021}\\
\cmidrule(lr){3-5}
\multirow{3}{=}[0.68em]{$\bullet$ For $C^{k-}_\ell$ where $\ell \not\equiv 0, \pm 1 \pmod k$ and $\gcd(k,\ell)=1$}
& & $p \ge 3$ & $0 < \gamma \le 1/3$ & \cite{S2025}\,/\!$\hspace{0.24cm}\bigstar\hspace{0.13cm}$ \\
\cmidrule(l){2-5}

% Sub-case 2
& \multirow{2}{3.5cm}[-0.2em]{Otherwise} 
& $p=2$ & $1/3 \le \gamma \le 1/2$ & $\hspace{0.24cm}\bigstar\hspace{0.13cm}$\,/\!\cite{HLS2021} \\
\cmidrule(lr){3-5} 
& & $p \ge 3$ & $\gamma = 1/3$ & $\hspace{0.24cm}\bigstar\hspace{0.13cm}$ \\

\bottomrule
\addlinespace
\multicolumn{5}{@{}p{\dimexpr\textwidth-2\tabcolsep}@{}}{%
\footnotesize Here, $p$ is the smallest prime factor of $k/\gcd(k,\ell)$ and $t$ is the smallest odd integer with $t \ge \max(3, \gcd(k,\ell))$. Prior results are cited accordingly, and new contributions established in this paper are marked with $\bigstar$. Citations for the lower and upper bounds are listed respectively, separated by a slash.} \\
\end{tabular}
\end{table}

\subsection{Organization.} 
The rest of the paper is organized as follows. In \cref{sec:2}, we introduce the Edge Type Framework and prove \cref{thm:new_lower_bound_revised} and \cref{thm:minus_edge_main}. This framework relies on \cref{thm:existence}, which is proved in \cref{sec:proof_of_existence}. In \cref{sec:4}, we prove \cref{thm:main_upper_bound} assuming \cref{thm:structural}. The proof of \cref{thm:structural} is then provided in \cref{sec:proof_of_structural}. Finally, we give some concluding remarks in \cref{sec:concluding}, followed by the proof of \cref{thm:minus_edge_main_2} in \hyperref[sec: appendix1]{the Appendix}.

\section{Lower Bound Constructions via $(k,\ell;d)$-Families}\label{sec:2}

In this section, we prove \cref{thm:new_lower_bound_revised,thm:minus_edge_main} using a general framework based on what we call \emph{edge types}.
We begin by reviewing the construction of Han, Lo, and Sanhueza-Matamala~\cite{HLS2021}, which yields the $1/p$ lower bound.
We then introduce the notion of edge types and use it to define a new combinatorial object, the $(k,\ell;d)$-family.
Finally, we show how the existence of such families leads to our lower bounds on the codegree Turán density.
The proof of the existence of these families (Theorem~\ref{thm:existence}) is deferred to Section~\ref{sec:proof_of_existence}.

\subsection{The HLS Construction and the \texorpdfstring{$1/p$}{1/p} Bound Revisited}\label{subsec:1/p}

Let us first present the construction of Han, Lo, and Sanhueza-Matamala. This construction serves as the basis for the more general framework we develop subsequently.

\begin{construction}[\!\!{\cite[Construction 10.1]{HLS2021}}]\label{cons:hls}
Let $k \ge 2$ and $d > 1$ be integers. For $n>0$, the $k$-graph $H_{n,d}^k$ is defined as follows. Given a vertex set $V$ of size $n$, partition it into $d$ disjoint vertex sets $V_1, \dots, V_d$ of sizes as equal as possible. Assume that every vertex in $V_i$ is labeled with $i \in \mathbb{Z}_d$. The edge set of $H_{n,d}^k$ consists of all $k$-subsets of $V$ for which the sum of the labels of their vertices is congruent to $1$ modulo $d$. \end{construction}
This construction yields the following lower bound. We present it in a slightly different form than the original proposition in~\cite{HLS2021}.\footnote{The original result (Proposition 10.1) in~\cite{HLS2021} is stated for an integer $d$ such that $d \mid k$ and $d \nmid \ell$, yielding a lower bound of $1/d$. The $1/p$ bound is a slightly stronger statement since the smallest prime factor $p$ of $k/\gcd(k,\ell)$ must be less than or equal to any such $d$.}

\begin{lemma}\label{lem:hls_general_version}
Let $2\le k< \ell$ satisfy  $k \nmid \ell$ and let $p$ be the smallest prime factor of $\frac{k}{\gcd(k,\ell)}$, then the hypergraph $H_{n, p}^k$ defined in Construction \ref{cons:hls} is $C_\ell^k$-free. Consequently,
$$\gamma(C^k_\ell)\geq \frac{1}{p}.$$
\end{lemma}
\begin{proof}
Suppose for contradiction that $H_{n, p}^k$ contains a copy of $C_\ell^k$, denoted by $C=(v_1, \dots, v_{\ell})$. For each $i \in [\ell]$, let $x_i \in [p]$ be the unique label such that $v_i\in V_{x_i}$. All indices are taken modulo $\ell$.

By the definition of an edge in $H_{n,p}^k$, we have $\sum_{j=1}^{k} x_{i+j} \equiv 1 \pmod{p}$ for all $i \in [\ell]$. Subtracting the equation for $i$ from that for $i-1$ yields $x_i \equiv x_{i+k} \pmod{p}$. Thus, the sequence of labels $(x_i)_{i\ge 1}$ is periodic with period $k$. As the cycle structure implies periodicity of period $\ell$, the sequence must also be periodic with period $\gcd(k,\ell)$.
This periodicity allows us to write the sum over the first $k$ labels as a multiple of the sum over the first $\gcd(k,\ell)$ labels,
$ \sum_{j=1}^{k} x_j = \frac{k}{\gcd(k,\ell)} \sum_{j=1}^{\gcd(k,\ell)} x_j. $
By hypothesis, $p$ divides $k/\gcd(k,\ell)$, so we have that
$$ 1 \equiv \sum_{j=1}^{k} x_j \equiv\frac{k}{\gcd(k,\ell)} \sum_{j=1}^{\gcd(k,\ell)} x_j \equiv0 \pmod{p}, $$
which leads to the final contradiction.
\end{proof}

\subsection{A New Construction Framework: from Edge-types to $(k,\ell;d)$-Families}

We now develop our general framework, which will be sufficient for constructing new lower bounds for both $C^k_\ell$ and $C^{k-}_\ell$.

For integers $d \ge 1$ and $k \ge 0$, let $\mathcal{T}_d^k$ denote the set of all ordered $d$-tuples of non-negative integers that sum to $k$, defined as:
\[
   \mathcal{T}_d^k = \left\{(x_1,\dots,x_d) \in \mathbb{Z}_{\ge 0}^d \colon \sum_{i=1}^d x_i=k\right\}.
\]

Using this notation, we introduce the following definitions.

\begin{definition}
Given a vertex set $V$ partitioned into $V_1, \dots, V_d$, we say a $k$-edge has \emph{edge type} $\vec{x}=(x_1,\dots,x_d) \in \mathcal{T}_d^k$ if it contains exactly $x_i$ vertices from $V_i$ for each $i \in [d]$.
Two types $\vec{x}, \vec{y} \in \mathcal{T}_d^k$ are \emph{adjacent} if $\|\vec{x}-\vec{y}\|_1=\sum_{i=1}^d|x_i-y_i|=2$.
For any family of types $\mathcal{T}\subseteq \mathcal{T}_d^k$, this adjacency relation defines a \emph{type graph} $G_{\mathcal{T}}$ whose vertex set is $\mathcal{T}$ and whose edges connect adjacent types. The \emph{connected components} of $\mathcal{T}$ are defined as the connected components of $G_{\mathcal{T}}$.
\end{definition}

We can now restate the construction of $H_{n, d}^k$ using a specific family of edge types. We term this the \emph{base family}, denoted by $\mathcal{B}_d^k$, as it serves as the basis for our later constructions. It is defined as:
$$ \mathcal{B}_d^k = \left\{ (x_1,\dots,x_d)\in\mathcal{T}_d^k \colon \sum_{j=1}^d j\cdot x_j\equiv 1 \pmod{d} \right\}. $$
Consequently, $H_{n, d}^k$ is simply the collection of all edges whose type belongs to $\mathcal{B}_d^k$.

We illustrate the base families $\mathcal{B}_3^3$ and $\mathcal{B}_3^5$ in the figure below. In these diagrams, the lattice points represent the complete set of edge types $\mathcal{T}_3^3$ and $\mathcal{T}_3^5$, and the grid lines connect adjacent types. The red nodes highlight the members of the base families, while the blue triangles visualize the extension from $(k-1)$-tuples to $k$-edges. For instance, the blue arrow in the left diagram identifies the triangle corresponding to $\vec{y}=(0,2,0)$, whose three vertices represent the types extending $\vec{y}$ (namely $(1,2,0)$, $(0,3,0)$, and $(0,2,1)$). Crucially, the red nodes are distributed so that every blue triangle is incident to at least one red vertex.

\begin{center}
\begin{tikzpicture}[scale=1.4]
    \foreach \y in {0,...,2} {
    \foreach \x in {\y,...,2} {
      \fill[blue!25] 
        (\x*1 - \y*0.5, \y*0.866) -- 
        (\x*1 - \y*0.5 + 1, \y*0.866) -- 
        (\x*1 - \y*0.5 + 0.5, \y*0.866 + 0.866) -- 
        cycle;
    }
  }
  \foreach \y in {0,...,3} {
    \draw[thick] (0.5*\y, {0.866*\y}) -- ({3 - 0.5*\y}, {0.866*\y});
    \ifnum \y<3
      \foreach \x in {\y,...,2} {
        \draw[thick] ({\x+1 - 0.5*\y}, {0.866*\y}) -- ({\x+1- 0.5*(\y+1)}, {0.866*(\y+1)});
      }
    \fi
    \ifnum \y>0
      \foreach \x in {\y,...,3} {
        \pgfmathtruncatemacro{\prevx}{\x - 1}
        \draw[thick] ({\x - 0.5*\y}, {0.866*\y}) -- ({\prevx - 0.5*(\y-1)}, {0.866*(\y-1)});
      }
    \fi
  }
  
  \foreach \y in {0,...,3} {
    \foreach \x in {\y,...,3} {
      \pgfmathtruncatemacro{\check}{mod(2*\y-\x+2, 3)}
      \ifnum \check = 0
        \draw[fill=red, thick] (\x*1 - \y*0.5, \y*0.866) circle (3.5pt);  % 满足条件时红色
      \else
        \draw[fill=white, thick] (\x*1 - \y*0.5, \y*0.866) circle (3.5pt);  % 不满足时灰色
      \fi
    }
  }

    \node[below] at (3,-0.1) {$(0,0,3)$};
    \node[above] at (1.5,2.598) {$(0,3,0)$};
    \node[below] at (0,-0.1) {$(3,0,0)$};
    \node[left] at (0.4,0.9) {$(2,1,0)$};
    \node[below] at (2,-0.1) {$(1,0,2)$};
    \node[below] at (1,-0.1) {$(2,0,1)$};
    \node[right] at (2.1,1.8) {$(0,2,1)$};
    \node[left] at (0.9,1.8) {$(1,2,0)$};
    \node[right] at (2.6,0.9) {$(0,1,2)$};
    \draw[->, thick, >=stealth, blue] (1.5,2.598*0.777) -- ++(0.8, 0.6) node[right] {$(0,2,0)$};
\end{tikzpicture}
\begin{tikzpicture}
    %添加小三角形
    \foreach \y in {0,...,4} {
    \foreach \x in {\y,...,4} {
      \fill[blue!25] 
        (\x*1 - \y*0.5, \y*0.866) -- 
        (\x*1 - \y*0.5 + 1, \y*0.866) -- 
        (\x*1 - \y*0.5 + 0.5, \y*0.866 + 0.866) -- 
        cycle;
    }
  }
  % 绘制所有三角形网格线
  \foreach \y in {0,...,5} {
    % 水平线
    \draw[thick] (0.5*\y, {0.866*\y}) -- ({5 - 0.5*\y}, {0.866*\y});
    % 左上到右下的斜线
    \ifnum \y<5
      \foreach \x in {\y,...,4} {
        \draw[thick] ({\x+1 - 0.5*\y}, {0.866*\y}) -- ({\x+1- 0.5*(\y+1)}, {0.866*(\y+1)});
      }
    \fi
    % 右上到左下的斜线
    \ifnum \y>0
      \foreach \x in {\y,...,5} {
        \pgfmathtruncatemacro{\prevx}{\x - 1}
        \draw[thick] ({\x - 0.5*\y}, {0.866*\y}) -- ({\prevx - 0.5*(\y-1)}, {0.866*(\y-1)});
      }
    \fi
  }
  
      \draw[thick] (0,0) -- (5*0.5,5*0.866) -- (5,0) -- cycle;  % 保持边框不变
      \foreach \y in {0,...,5} {
        \foreach \x in {\y,...,5} {
          \pgfmathtruncatemacro{\check}{mod(2*\y-\x+1, 3)}
          \ifnum \check = 0
            \draw[fill=red, thick] (\x*1 - \y*0.5, \y*0.866) circle (4pt);  % 满足条件时红色
          \else
            \draw[fill=white, thick] (\x*1 - \y*0.5, \y*0.866) circle (4pt);  % 不满足时灰色
          \fi
        }
      }
        % 添加标签
    \node[below right] at (5,-0.1) {$(0,0,5)$};
    \node[below] at (1,-0.1) {$(4,0,1)$};
    \node[below] at (4,-0.1) {$(1,0,4)$};
    \node[above] at (2.5,5*0.866) {$(0,5,0)$};
    \node[below left] at (0,-0.1) {$(5,0,0)$};
    \node[left,xshift=-1mm] at (0.5*5-0.5,5*0.866-0.866) {$(1,4,0)$};
    \node[left,xshift=-1mm] at (0.5*5-0.5-0.5-0.5,5*0.866-0.866-0.866-0.866) {$(3,2,0)$};
    \node[right,xshift=1mm] at (0.5*5+0.5,5*0.866-0.866) {$(0,4,1)$};
    \node[right,xshift=1mm] at (0.5*5+0.5+0.5+0.5,5*0.866-0.866-0.866-0.866) {$(0,2,3)$};
    \draw[->, thick, >=stealth, blue] (2.5, 3.75) -- ++(0.8, 0.4) node[right] {$(0,4,0)$};
\end{tikzpicture}

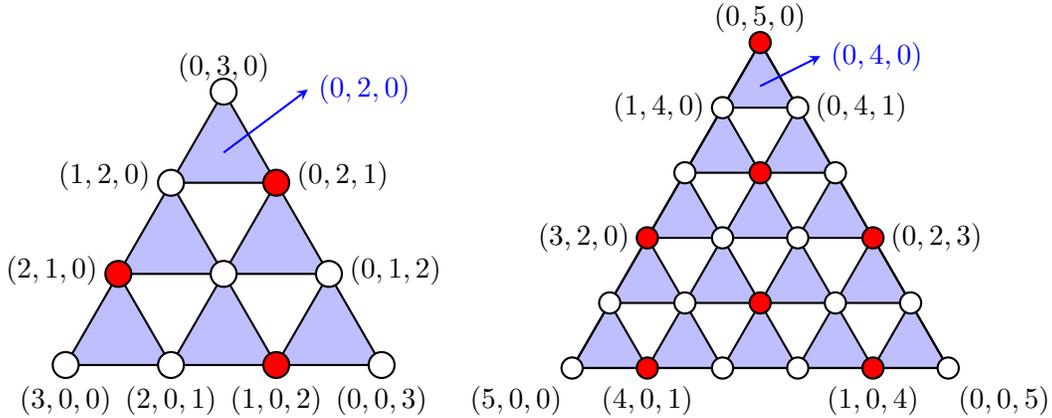
\captionof{figure}{Visual illustration of the base families $\mathcal{B}_3^3$ and $\mathcal{B}_3^5$}
\end{center}

The success of the above $H_{n, d}^k$ construction suggests a general new approach as follows: We first identify a ``well-behaved'' family of types $\mathcal{T}\subseteq \mathcal{T}_d^k$, and then define our hypergraph $H$ as the one comprising all edges whose type belongs to $\mathcal{T}$. 
The main task is thus to formalize the properties that make a family of types $\mathcal{T}$ ``well-behaved'', which must satisfy two crucial properties. First, to ensure a large minimum $(k-1)$-degree, we require that any set of $k-1$ vertices can be extended to an edge by adding a vertex from some partition $V_j$. Second, we impose a constraint on $\mathcal{T}$ to make the resulting hypergraph free of a tight cycle $C^k_\ell$ or a tight cycle with a missing edge $C^{k-}_\ell$.

This motivates the following definition of a $(k,\ell;d)$-family.

\begin{definition}\label{def:kld_family}
A family of types $\mathcal{T}\subseteq \mathcal{T}_d^k$ is a \emph{$(k,\ell;d)$-family} if it satisfies the following:
\begin{enumerate}
    \item[(P1).] For any tuple $\vec{y}= (y_1, \dots, y_d) \in \mathcal{T}_d^{k-1}$, there exists a type $\vec{x}= (x_1, \dots, x_d) \in \mathcal{T}\subseteq \mathcal{T}_d^k$ such that $\|\vec{x} - \vec{y}\|_1 = \sum_{i=1}^d |x_i - y_i| = 1$, or equivalently $\vec{y}$ is obtained from $\vec{x}$ by decreasing exactly one coordinate by 1.
    \item[(P2).] For every connected component $C$ of $G_{\mathcal{T}}$, there exists an index set $I_C \subseteq [d]$ and an integer $v_C$ such that $\frac{k}{\gcd(k,\ell)} \nmid v_C$ and for all $\vec{x} = (x_1, \dots, x_d) \in C$, we have $\sum_{i \in I_C} x_i = v_C$.
\end{enumerate}
\end{definition}

The following theorem asserts that under reasonable arithmetic assumptions on $k$ and $\ell$, such $(k,\ell;d)$-families exist,
with proof deferred to \cref{sec:proof_of_existence}. 

\begin{theorem}\label{thm:existence}
Let $3\le k< \ell$ be integers with $k \nmid \ell$. Let $p$ be the smallest prime factor of $\frac{k}{\gcd(k,\ell)}$, and let $t$ be the smallest odd integer satisfying $t \ge \max(3, \gcd(k,\ell))$. Then,
\begin{enumerate}
    \item A $(k,\ell;p)$-family exists.
    \item If $\frac{k}{\gcd(k,\ell)}\geq 5$, a $(k,\ell;t)$-family exists.
\end{enumerate}
\end{theorem}

\subsection{Lower Bounds from $(k,\ell;d)$-Families}

We now introduce the following key theorem, which shows that the existence of $(k,\ell;d)$-families is sufficient to establish our desired lower bounds.

\begin{theorem}\label{thm:lower_bound_part1}
If there exists a $(k,\ell;d)$-family $\mathcal{T}$, then $$\gamma(C^k_\ell)\geq \frac{1}{d}.$$
Furthermore if $\gcd(k,\ell)>1$, then $$\gamma(C^{k-}_\ell)\geq \frac{1}{d}.$$
\end{theorem}
\begin{proof}
Let $H$ be the $k$-uniform hypergraph constructed on a sufficiently large $n$-vertex set $V$ from the given $(k,\ell;d)$-family $\mathcal{T}$. The construction partitions $V$ into $d$ nearly equal-sized sets $V_1, \dots, V_d$, and $E(H)$ consists of all $k$-sets whose type is in $\mathcal{T}$.

We first prove that $H$ is $C^k_\ell$-free by contradiction.  Assume $H$ contains a copy of $C^k_\ell$ on the vertices $(v_1, \dots, v_\ell)$ then $\{v_i, \dots, v_{i+k-1}\}$ is an edge for all $i\in[\ell]$, where all indices are modulo $\ell$. The tight structure of $C^k_\ell$ implies that the types of any two consecutive edges are either identical or adjacent. Therefore, all such edges must have types belonging to a single connected component $C$ in $G_{\mathcal{T}}$.  As $\mathcal{T}$ is a $(k,\ell;d)$-family, this component $C$ is associated with a non-empty index set $I_C \subseteq [d]$ and an integer invariant $v_C$. For $j\in[\ell]$, let $y_j=1$ if $v_j \in \bigcup_{m \in I_C} V_m$, and $y_j=0$ if $v_j \notin \bigcup_{m \in I_C} V_m$. Then for any $t \in [\ell]$,
$$  \sum_{j=t}^{t+k-1} y_j=v_C . $$

Summing this for all $t\in [\ell]$ gives $$\ell v_C=\sum_{t\in[\ell]}\sum_{j=t}^{t+k-1} y_j=k\sum_{j=1}^{\ell} y_j.$$ This means $\frac{\ell}{\gcd(k,\ell)}\cdot v_C =\frac{k}{\gcd(k,\ell)}\cdot \sum_{j=1}^{\ell} y_j$. As $\frac{k}{\gcd(k,\ell)}$ and $\frac{\ell}{\gcd(k,\ell)}$ are coprime, $v_C$ must be divisible by $ \frac{k}{\gcd(k,\ell)}$, a contradiction.

If $\gcd(k,\ell)>1$, assume $H$ contains a copy of $C^{k-}_\ell$ on the vertices $(v_1, \dots, v_\ell)$ and that  $\{v_i, \dots, v_{i+k-1}\}$ is an edge for all $i\in[\ell]\setminus \{1\}$ where all indices are modulo $\ell$. Contradiction can be derived as before, provided we can first show that $\sum_{j=1}^{k} y_j=v_C$ as well. We let $z_i=\sum_{j=i}^{i+k-1} y_j$, then $z_i=v_C$ for all $i\in[\ell]\setminus \{1\}$. The following equation shows that $\sum_{j=1}^{k} y_j=z_1=v_C$.
$$\frac{\ell}{\gcd(k,\ell)}\cdot v_C=\sum_{\substack{i\in[\ell] \\ \gcd(k,\ell)\mid i}}z_i=\frac{k}{\gcd(k,\ell)}\cdot \sum_{j=1}^{\ell} y_j=\sum_{\substack{i\in[\ell] \\ \gcd(k,\ell)\mid i-1}}z_i=(\frac{\ell}{\gcd(k,\ell)}-1)\cdot v_C +z_1.$$
Note that the first equality relies on $\gcd(k,\ell) > 1$, which excludes index $1$ from the summation.
\end{proof}

\cref{thm:new_lower_bound_revised} and \cref{thm:minus_edge_main} now follow immediately by combining two key results: \cref{thm:existence}, which guarantees the existence of specific $(k,\ell;d)$-families, and \cref{thm:lower_bound_part1}, which converts their existence into lower bounds on the codegree Turán density.

\begin{proof}[Proof of \cref{thm:new_lower_bound_revised} and \cref{thm:minus_edge_main}]
If $p \in \{2,3\}$, by \cref{thm:existence}, a $(k,\ell;p)$-family exists. Thus, by \cref{thm:lower_bound_part1}, $\gamma(C^k_\ell)\geq \frac{1}{p}= \max( \frac{1}{p},\frac{1}{t})$, and  $\gamma(C^{k-}_\ell)\geq \frac{1}{p}= \max( \frac{1}{p},\frac{1}{t})$ when $\gcd(k,\ell) > 1$.

If $p > 3$, then $\frac{k}{\gcd(k,\ell)} \geq p \geq 5$. By \cref{thm:existence}, a $(k,\ell; p)$-family and a $(k,\ell; t)$-family exist. Thus, by \cref{thm:lower_bound_part1}, $\gamma(C^k_\ell)\geq \max( \frac{1}{p},\frac{1}{t})$, and  $\gamma(C^{k-}_\ell)\geq\max( \frac{1}{p},\frac{1}{t})$ when $\gcd(k,\ell) > 1$.\end{proof}

\section{Proof of \cref{thm:existence}}
\label{sec:proof_of_existence}
This section is devoted to the proof of \cref{thm:existence}. We construct the required families and verify that they satisfy the conditions of Definition \ref{def:kld_family}. Throughout the section, we let  $\vec{e}_i$ denote the vector in $\mathbb{R}^d$ with 1 in the $i$-th coordinate and 0 in all other coordinates. Our construction relies on the base family
\[\mathcal{B}_d^k =\left\{ (x_1,\dots,x_d)\in\mathcal{T}_d^k \colon \sum_{j=1}^d j\cdot x_j\equiv 1 \pmod{d} \right\}.\]
The following claim establishes two fundamental properties of this family.
\begin{claim}\label{claim:base_family_props}
Let $k,\ell,d$ be positive integers. The family $\mathcal{B}_d^k$ satisfies the following:
\begin{enumerate}
    \item For any $\vec{y} \in \mathcal{T}_d^{k-1}$, there exists a type $\vec{x} \in \mathcal{B}_d^k$ such that $\vec{x} = \vec{y} + \vec{e}_j$ for some $j \in [d]$.
    \item The connected components of $\mathcal{B}_d^k$ are isolated vertices.
\end{enumerate}
\end{claim}

\begin{poc}
For the first property, given any $\vec{y} \in \mathcal{T}_d^{k-1}$, let $j \in [d]$ be the unique index satisfying $j \equiv 1 - \sum_{i=1}^d i y_i \pmod d$. Define $\vec{x} = \vec{y} + \vec{e}_j\in \mathcal{T}_d^k$. Then $\sum_{i=1}^d i x_i = \sum_{i=1}^d i y_i + j \equiv (\sum_{i=1}^d i y_i) + 1 - (\sum_{i=1}^d i y_i) \equiv 1 \pmod d$. Thus, $\vec{x} \in \mathcal{B}_d^k$.

For the second property, assume for contradiction that two distinct types $\vec{x}, \vec{x}' \in \mathcal{B}_d^k$ are adjacent. We can then write $\vec{x}' = \vec{x} - \vec{e}_{i_1} + \vec{e}_{i_2}$ for distinct $i_1, i_2 \in [d]$. Their defining sums satisfy $(\sum i x_i) - (\sum i x'_i) = i_1 - i_2 $. However, $(\sum i x_i) - (\sum i x'_i)\equiv 1-1 \equiv 0 \pmod d$. This implies $i_1 \equiv i_2 \pmod d$, which contradicts the fact that $i_1$ and $i_2$ are distinct elements of $[d]$.
\end{poc}

With these properties, we begin by showing that the type family $\mathcal{B}_p^k$,  which corresponds to Construction \ref{cons:hls} introduced by Han, Lo, and Sanhueza-Matamala~\cite{HLS2021}, is a $(k,\ell; p)$-Family

\subsection{Constructing a $(k,\ell; p)$-Family via Direct Algebraic Construction}

\begin{proof}[Proof of Theorem~\ref{thm:existence} (1)]
Let $p$ be the smallest prime factor of $\frac{k}{\gcd(k,\ell)}$. We claim that $\mathcal{B}_p^k =\{ (x_1,\dots,x_p)\in\mathcal{T}_p^k \colon \sum_{j=1}^p j\cdot x_j\equiv 1 \pmod{p} \}$
is a $(k,\ell; p)$-family. To establish this, we verify the two properties from Definition \ref{def:kld_family}.

By Claim \ref{claim:base_family_props}, for any $\vec{y}=(y_1, \dots, y_p) \in \mathcal{T}_p^{k-1}$, there exists a type $\vec{x}=(x_1, \dots, x_p) \in \mathcal{B}_p^k$ such that $\vec{x} = \vec{y} + \vec{e}_j$ for some $j \in [p]$ and thus $\|\vec{x} - \vec{y}\|_1 = \sum_{i=1}^p |x_i - y_i| = 1$.

By Claim \ref{claim:base_family_props}, $G_{\mathcal{T}}$ consists of isolated vertices. 
Since each component $C$ is a singleton $\{\vec{x}=(x_1, \dots, x_p)\}$, we must find a suitable invariant. The defining congruence $\sum i x_i \equiv 1 \pmod p$ ensures that at least one coordinate $x_j$ is not divisible by $p$, as the sum would otherwise be divisible by $p$. We can therefore set the index set $I_C = \{j\}$ and the value $v_C = x_j$. Because $p$ is a prime factor of $\frac{k}{\gcd(k,\ell)}$ and $p$ does not divide $v_C$, it follows that $\frac{k}{\gcd(k,\ell)}$ cannot divide $v_C$.

With both properties satisfied, the proof is complete. \end{proof}

\subsection{Constructing a $(k,\ell; t)$-Family via Local Replacement}
\label{subsec:local_replacement}
The proof of the second statement relies on a local replacement method, which we outline below. Recall that $p$ is the smallest prime factor of $k/\gcd(k,\ell)$, and by assumption, $k/\gcd(k,\ell) \ge 5$.

Our method begins with a base family $\mathcal{B}_d^k$. This family is constructed such that all its connected components are singletons, and satisfies the first property in Definition \ref{def:kld_family}. However, some of these singleton components violate the second property in Definition \ref{def:kld_family}. Let $\mathcal{P}$ be the set of these problematic types. A type $\vec{x}$ is in $\mathcal{P}$ if and only if every coordinate of $\vec{x}$ is a multiple of $k/\gcd(k,\ell)$. The condition $k/\gcd(k,\ell) \ge 5$ ensures these problematic types are well-separated.

Our construction of the final family proceeds by first removing all types in $\mathcal{P}$ from $\mathcal{B}_d^k$. This action resolves the issue of the second property in Definition \ref{def:kld_family} but disrupts the first property in Definition \ref{def:kld_family} at those locations. Then, for each type $\vec{x} \in \mathcal{P}$, we introduce a corresponding replacement set, $\mathcal{R}_{\vec{x}}$. This set is carefully chosen to restore the first property in Definition \ref{def:kld_family}  locally and to ensure that the new connected components formed by types in $\mathcal{R}_{\vec{x}}$ satisfy the second property in Definition \ref{def:kld_family}. The fact that the types in $\mathcal{P}$ are well-separated guarantees these local replacements do not interfere with one another.

The resulting family $\mathcal{T}$ is therefore defined as:
$$\mathcal{T} = \left( \mathcal{B}_d^k\setminus \mathcal{P} \right) \cup \left( \bigcup_{{\vec{x}} \in \mathcal{P}} \mathcal{R}_{\vec{x}}\right).$$

To make the method just described more concrete, we now consider the case where $k=5$ and $5 \nmid \ell$, and construct a corresponding $(k,\ell;3)$-family. This process is shown in the following three figures. First, the left figure shows the base family $\mathcal{B}_3^5 = \{(x_1,x_2,x_3) \in \mathcal{T}_3^5 \mid x_1+2x_2 \equiv 1 \pmod{3}\}$. Here the set of problematic types $\mathcal{P}$ consists of a single element ${\vec{v}}=(0,5,0)$. The middle figure shows the family after $v$ has been removed. This removal leaves the tuple $(0,4,0)$ uncovered, and the red triangle highlights this gap, as its vertices represent the types $(0,5,0), (1,4,0),$ and $(0,4,1)$, which are required to cover the tuple $(0,4,0)$. To remedy this, we introduce the replacement set $\mathcal{R}_{\vec{v}}= \{(1,4,0)\}$, yielding the final family $\mathcal{T} = (\mathcal{B}_3^5 \setminus \{{\vec{v}}\}) \cup \mathcal{R}_{\vec{v}}$ shown on the right. In this final family, the new connected component $C = \{(1,4,0), (1,3,1)\}$ satisfies the second property in Definition \ref{def:kld_family} with the invariant set $I_C=\{1\}$ and value $v_C=1$.

\begin{figure}[H]
\begin{minipage}{0.28\textwidth}
    \centering
    % Figure 1: Base Family
    \begin{tikzpicture}[scale=0.9]
      \def\k{5}
      % Background fill
      \foreach \y in {0,...,\numexpr\k-1} {
        \foreach \x in {\y,...,\numexpr\k-1} {\fill[blue!25] (\x - 0.5*\y, \y*0.866) -- (\x+1 - 0.5*\y, \y*0.866) -- (\x+0.5 - 0.5*\y, \y*0.866+0.866) -- cycle;}
      }
      % Grid lines
      \foreach \y in {0,...,\k} {
        \draw[thick] (0.5*\y, {0.866*\y}) -- ({\k - 0.5*\y}, {0.866*\y});
        \ifnum \y<\k \foreach \x in {\y,...,\numexpr\k-1} {\draw[thick] ({\x+1 - 0.5*\y}, {0.866*\y}) -- ({\x+1- 0.5*(\y+1)}, {0.866*(\y+1)});} \fi
        \ifnum \y>0 \foreach \x in {\y,...,\k} {\draw[thick] ({\x - 0.5*\y}, {0.866*\y}) -- ({\x - 1 - 0.5*(\y-1)}, {0.866*(\y-1)});} \fi
      }
      % Nodes
      \foreach \y in {0,...,\k} {
        \foreach \x in {\y,...,\k} {
          \pgfmathtruncatemacro{\check}{mod(\k - \x + 2*\y - 1, 3)}
          \ifnum \check = 0 \draw[fill=red, thick] (\x - \y*0.5, \y*0.866) circle (4.5pt);
          \else \draw[fill=white, thick, draw=black] (\x - \y*0.5, \y*0.866) circle (4.5pt); \fi
        }
      }
      \node[below, font=\scriptsize] at (\k,0) {$(0,0,5)$};
      \node[above,yshift=2mm] at (0.5*\k, \k*0.866) {$(0,5,0)$};
      \node[below, font=\scriptsize] at (0,0) {$(5,0,0)$};
    \node[left,xshift=-1mm] at (0.5*\k-0.5,\k*0.866-0.866) {$(1,4,0)$};
    \node[right,xshift=1mm] at (0.5*\k+0.5,\k*0.866-0.866) {$(0,4,1)$};
    \draw[->, thick, >=stealth, blue] (2.5, 3.75) -- ++(1.2, 0.8) node[right] {$(0,4,0)$};
    % 这里的 \vphantom 让这个短公式强行占用了长公式的高度
    \node[below] at (2.5, -0.5) {\large $\mathcal{B}_3^5 \vphantom{(\mathcal{B}_3^5 \setminus \{{\vec{v}}\}) \cup \mathcal{R}_{\vec{v}}}$};
    \end{tikzpicture}
\end{minipage}\hspace{4mm} {\Large$\rightarrow$ }\hspace{-5mm}
\begin{minipage}{0.28\textwidth}
    \centering
    % Figure 2: Problematic type removed, gap highlighted
    \begin{tikzpicture}[scale=0.9]
      \def\k{5}
      % Background fill
      \foreach \y in {0,...,\numexpr\k-1} {
        \foreach \x in {\y,...,\numexpr\k-1} {\fill[blue!25] (\x - 0.5*\y, \y*0.866) -- (\x+1 - 0.5*\y, \y*0.866) -- (\x+0.5 - 0.5*\y, \y*0.866+0.866) -- cycle;}
      }
      % Highlight the extension gap for (0,4,0) which is the triangle formed by
      % nodes (0,4,0), (1,4,0), and (0,5,0).
      % TikZ coordinates for these are (5,4), (4,4), and (5,5) respectively.
      \fill[red!50] (5-0.5*4, 4*0.866) -- (4-0.5*4, 4*0.866) -- (5-0.5*5, 5*0.866) -- cycle;
      % Grid lines
      \foreach \y in {0,...,\k} {
        \draw[thick] (0.5*\y, {0.866*\y}) -- ({\k - 0.5*\y}, {0.866*\y});
        \ifnum \y<\k \foreach \x in {\y,...,\numexpr\k-1} {\draw[thick] ({\x+1 - 0.5*\y}, {0.866*\y}) -- ({\x+1- 0.5*(\y+1)}, {0.866*(\y+1)});} \fi
        \ifnum \y>0 \foreach \x in {\y,...,\k} {\draw[thick] ({\x - 0.5*\y}, {0.866*\y}) -- ({\x - 1 - 0.5*(\y-1)}, {0.866*(\y-1)});} \fi
      }
      % Nodes
      \foreach \y in {0,...,\k} {
        \foreach \x in {\y,...,\k} {
        \pgfmathtruncatemacro{\check}{mod(\k - \x + 2*\y - 1, 3)}
          \ifnum\x=5 \ifnum\y=5 \def\check{1} \fi\fi % Force (0,5,0) to be white
          \ifnum \check = 0 \draw[fill=red, thick] (\x - \y*0.5, \y*0.866) circle (4.5pt);
          \else \draw[fill=white, thick, draw=black] (\x - \y*0.5, \y*0.866) circle (4.5pt); \fi
        }
      }
      \draw[red, dashed, thick] (5 - 5*0.5, 5*0.866) circle (8pt);
      \node[below, font=\scriptsize] at (\k,0) {$(0,0,5)$};
      \node[above,yshift=2mm] at (0.5*\k, \k*0.866) {$(0,5,0)$};
      \node[below, font=\scriptsize] at (0,0) {$(5,0,0)$};
    \node[left,xshift=-1mm] at (0.5*\k-0.5,\k*0.866-0.866) {$(1,4,0)$};
    \node[right,xshift=1mm] at (0.5*\k+0.5,\k*0.866-0.866) {$(0,4,1)$};
    \draw[->, thick, >=stealth, red] (2.5, 3.75) -- ++(1.2, 0.8) node[right] {$(0,4,0)$};
    \node[below] at (2.5, -0.5) {\large $\mathcal{B}_3^5 \setminus \{{\vec{v}}\}$};
    \end{tikzpicture}
\end{minipage}\hspace{4mm} {\Large$\rightarrow$ }\hspace{-5mm}
\begin{minipage}{0.28\textwidth}
    \centering
    % Figure 3: Final family with replacement
    \begin{tikzpicture}[scale=0.9] 
      \def\k{5}
      % Background fill
      \foreach \y in {0,...,\numexpr\k-1} {
        \foreach \x in {\y,...,\numexpr\k-1} {\fill[blue!25] (\x - 0.5*\y, \y*0.866) -- (\x+1 - 0.5*\y, \y*0.866) -- (\x+0.5 - 0.5*\y, \y*0.866+0.866) -- cycle;}
      }
      % Grid lines
      \foreach \y in {0,...,\k} {
        \draw[thick] (0.5*\y, {0.866*\y}) -- ({\k - 0.5*\y}, {0.866*\y});
        \ifnum \y<\k \foreach \x in {\y,...,\numexpr\k-1} {\draw[thick] ({\x+1 - 0.5*\y}, {0.866*\y}) -- ({\x+1- 0.5*(\y+1)}, {0.866*(\y+1)});} \fi
        \ifnum \y>0 \foreach \x in {\y,...,\k} {\draw[thick] ({\x - 0.5*\y}, {0.866*\y}) -- ({\x - 1 - 0.5*(\y-1)}, {0.866*(\y-1)});} \fi
      }
      % Nodes
      \foreach \y in {0,...,\k} {
        \foreach \x in {\y,...,\k} {
          \pgfmathtruncatemacro{\check}{mod(\k - \x + 2*\y - 1, 3)}
          \ifnum\x=5 \ifnum\y=5 \def\check{1} \fi\fi % Force (0,5,0) to be white
          \ifnum\x=4 \ifnum\y=4 \def\check{0} \fi\fi % Force (1,4,0) to be red
          \ifnum \check = 0 \draw[fill=red, thick] (\x - \y*0.5, \y*0.866) circle (4.5pt);
          \else \draw[fill=white, thick, draw=black] (\x - \y*0.5, \y*0.866) circle (4.5pt); \fi
        }
      }
      \draw[red, dashed, thick] (5 - 5*0.5, 5*0.866) circle (8pt);
      \draw[red, solid, thick] (4 - 4*0.5, 4*0.866) circle (8pt);
      \node[below, font=\scriptsize] at (\k,0) {$(0,0,5)$};
      \node[above,yshift=2mm] at (0.5*\k, \k*0.866) {$(0,5,0)$};
      \node[below, font=\scriptsize] at (0,0) {$(5,0,0)$};
    \node[left,xshift=-1mm] at (0.5*\k-0.5,\k*0.866-0.866) {$(1,4,0)$};
    \node[right,xshift=1mm] at (0.5*\k+0.5,\k*0.866-0.866) {$(0,4,1)$};
    \draw[->, thick, >=stealth, blue] (2.5, 3.75) -- ++(1.2, 0.8) node[right] {$(0,4,0)$};
    \node[below] at (2.5, -0.5) {\large $(\mathcal{B}_3^5 \setminus \{{\vec{v}}\}) \cup \mathcal{R}_{\vec{v}}$};
    \end{tikzpicture}
\end{minipage}
\caption{Construction of a $(k,\ell;3)$-family for $k=5, 5 \nmid \ell$.}
\label{fig:replacement_illustration}
\end{figure}
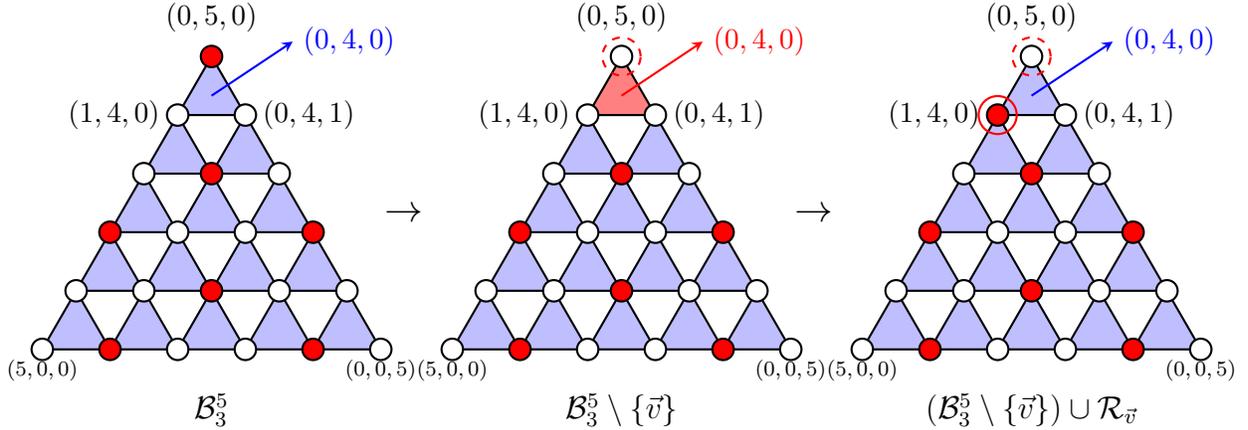

We now provide the formal proof of the second statement  of Theorem~\ref{thm:existence}. 

\begin{proof}[Proof of Theorem~\ref{thm:existence} (2)]
Let $q = \frac{k}{\gcd(k,\ell)}$. By assumption, we have $q \ge 5$. Set $d=t$, where $t$ is the smallest odd integer such that $t \ge \max(3, \gcd(k,\ell))$. 

First, we begin with the base family  $\mathcal{B}_d^k= \left\{ \vec{x} = (x_1, \dots, x_d) \in \mathcal{T}_d^k \colon \sum_{i=1}^d i \cdot x_i \equiv 1 \pmod d \right\}$. By  Claim \ref{claim:base_family_props}, $\mathcal{B}_d^k$ satisfies  the first property in Definition \ref{def:kld_family} and the connected components of $\mathcal{B}_d^k$ are isolated vertices.

Now we modify this base family to satisfy the second property in Definition \ref{def:kld_family}. We first define the set of problematic types as
\[ \mathcal{P} := \left\{ \vec{x} = (x_1, \dots, x_d) \in \mathcal{B}_d^k \colon \forall i \in [d], q \mid x_i \right\}. \]
Any type $\vec{x} \in \mathcal{B}_d^k \setminus \mathcal{P}$ has at least one coordinate $x_j$ not divisible by $q$. Its component is the singleton $C=\{\vec{x}\}$, for which we can choose $I_C = \{j\}$ and $v_C = x_j$, satisfying  the second property in Definition \ref{def:kld_family} as $q \nmid v_C$. Thus, we only need to consider the types in $\mathcal{P}$.

\begin{claim}\label{claim:zero_component}
For any type $\vec{x}=(x_1,\dots,x_d) \in \mathcal{P}$, there exists an index $\alpha \in [d]$ with $x_\alpha=0$.
\end{claim}
\begin{poc}
Assume for contradiction that $\vec{x} \in \mathcal{P}$ has $x_i > 0$ for all $i \in [d]$. Since $\vec{x} \in \mathcal{P}$, each $x_i$ is a multiple of $q$, so $x_i \ge q$. This gives the inequality
\[ k = \sum_{i=1}^d x_i \ge \sum_{i=1}^d q = d \cdot q  \ge \gcd(k,\ell) \cdot \frac{k}{\gcd(k,\ell)} = k. \]
This forces all inequalities to be equalities, which implies $x_i = q$ for all $i \in [d]$. As $\vec{x}$ is in $\mathcal{P} \subseteq \mathcal{B}_d^k$, it must satisfy the condition $\sum_{i=1}^d i x_i \equiv 1 \pmod d$. However, since $d$ is odd by definition, $ \sum_{i=1}^d i x_i =  \sum_{i=1}^d i\cdot q =  \frac{d(d+1)}{2}\cdot q= (q\cdot\frac{d+1}{2})\cdot d \equiv 0 \not\equiv 1\pmod d $,  which is a contradiction.
\end{poc}

For each $\vec{x} \in \mathcal{P}$, we now define a replacement set $\mathcal{R}_{\vec{x}}$. By Claim~\ref{claim:zero_component}, we can fix an index $\alpha \in [d]$ such that $x_\alpha=0$. All indices are considered modulo $d$. We define
\[ \mathcal{R}_{\vec{x}} := \left( \left\{ \vec{x} - \vec{e}_i + \vec{e}_{i-1} : i \in [d] \setminus \{\alpha, \alpha+1\} \right\} \cup \left\{ \vec{x} - \vec{e}_{\alpha+1} + \vec{e}_{\alpha-1} \right\} \right) \cap \mathcal{T}_d^k. \]

Let us first verify that  the first property in Definition \ref{def:kld_family} is restored. When a type $\vec{x} \in \mathcal{P}$ is removed, the tuples of the form $\vec{y} = \vec{x} - \vec{e}_j\in \mathcal{T}_d^{k-1}$ lose their covering type $\vec{x}$. The claim below will establish that for each such tuple $\vec{y}$, there is a new type in the replacement set $\mathcal{R}_{\vec{x}}$ which covers it.

\begin{claim}\label{claim:extension_restored}
For $\vec{x} \in \mathcal{P}$, if a tuple $\vec{y} \in \mathcal{T}_d^{k-1}$ satisfies $\vec{y} = \vec{x} - \vec{e}_j$ for some index $j \in [d]$, then there exists a type $\vec{x}' \in \mathcal{R}_{\vec{x}}$ such that $\vec{y} = \vec{x}' - \vec{e}_{j'}$ for some index $j' \in [d]$.
\end{claim}

\begin{poc}
All indices are considered modulo $d$. Let $\vec{x} \in \mathcal{P}$ and $\vec{y} = \vec{x} - \vec{e}_j$ for some $j \in [d]$. By Claim~\ref{claim:zero_component}, there exists an index $\alpha \in [d]$ such that $x_\alpha=0$. Since $x_j$ must be positive for $\vec{y}$ to be a valid tuple, we have $j \neq \alpha$.  We choose $\vec{x}'$ and $j'$ based on the value of $j$ as follows:
\begin{itemize}
    \item If $j=\alpha+1$, we choose $\vec{x}' = \vec{x} - \vec{e}_{\alpha+1} + \vec{e}_{\alpha-1}$ and $j'=\alpha-1$.
    \item If $j \in [d] \setminus \{\alpha, \alpha+1\}$, we choose $\vec{x}' = \vec{x} - \vec{e}_j + \vec{e}_{j-1}$ and $j'=j-1$.
\end{itemize} In both cases, $\vec{x}' \in \mathcal{R}_{\vec{x}}$ and $\vec{x}' - \vec{e}_{j'} = \vec{y}$.
\end{poc}

We aim to show the following family is a $(k,\ell;d)$-family:
$$\mathcal{T} := \left( \mathcal{B}_d^k \setminus \mathcal{P} \right) \cup \left( \bigcup_{\vec{x} \in \mathcal{P}} \mathcal{R}_{\vec{x}} \right).$$

By Claim \ref{claim:base_family_props} and Claim \ref{claim:extension_restored}, this family $\mathcal{T}$ satisfies the first property in Definition \ref{def:kld_family}.
It remains to verify that for each connected component $C$ of the type-graph $G_{\mathcal{T}}$, there exists an index set $I_C \subseteq [d]$ and an integer $v_C$ with $q \nmid v_C$ such that $\sum_{i \in I_C} z_i = v_C$ for all $\vec{z} \in C$.

We now analyze the connected components of $\mathcal{T}$. We will eventually show that the connected components of $\mathcal{T}$ are either a single type from $\mathcal{B}_d^k \setminus \mathcal{P}$, or contain exactly one type from $\bigcup_{\vec{x} \in \mathcal{P}} \mathcal{R}_{\vec{x}}$ and contains its neighbor, forming a star-like structure. We first show that the replacement types in $\bigcup_{\vec{x} \in \mathcal{P}} \mathcal{R}_{\vec{x}}$ lie in distinct connected components.

\begin{claim}\label{claim:component_separation}
Any two distinct types in $\bigcup_{\vec{x} \in \mathcal{P}} \mathcal{R}_{\vec{x}}$ belong to different connected components of $\mathcal{T}$.
\end{claim}

\begin{poc}
Assume for contradiction that the claim is false. Let $\{\vec{z}_1, \vec{z}_2\}$ be a pair of distinct types in $\bigcup_{\vec{x} \in \mathcal{P}} \mathcal{R}_{\vec{x}}$ that lie in the same connected component and for which their distance in the graph $G_{\mathcal{T}}$, denoted $d_{G_{\mathcal{T}}}(\vec{z}_1, \vec{z}_2)$, is minimized. By the minimality of $d_{G_{\mathcal{T}}}(\vec{z}_1, \vec{z}_2)$, any intermediate vertex on a shortest path $P$ between them cannot belong to $\bigcup_{\vec{x} \in \mathcal{P}} \mathcal{R}_{\vec{x}}$, and must therefore lie in $\mathcal{B}_d^k$.  Since Claim~\ref{claim:base_family_props} establishes that all types in $\mathcal{B}_d^k$ are isolated from one another, such a path cannot contain two or more consecutive vertices from $\mathcal{B}_d^k$. Thus, $P$ contains at most one intermediate vertex, limiting its length to 1 or 2. The remainder of the proof is dedicated to showing that $d_{G_{\mathcal{T}}}(\vec{z}_1, \vec{z}_2) > 2$, which makes such a path impossible, yielding the desired contradiction.

First, consider the case where $\vec{z}_1 \in \mathcal{R}_{\vec{x}_1}$ and $\vec{z}_2 \in \mathcal{R}_{\vec{x}_2}$ for distinct $\vec{x}_1, \vec{x}_2 \in \mathcal{P}$. By construction, all coordinates of vectors in $\mathcal{P}$ are multiples of $q \geq 5$. This ensures that for any distinct $\vec{x}_1, \vec{x}_2 \in \mathcal{P}$, their $L_1$-distance $\|\vec{x}_1 - \vec{x}_2\|_1$ is at least $2q \geq 10$. The graph distance is bounded below by half the $L_1$-distance, as each edge spans an $L_1$-distance of 2. We derive:
$$d_{G_{\mathcal{T}}}(\vec{z}_1, \vec{z}_2) \ge \frac{1}{2} \|\vec{z}_1 - \vec{z}_2\|_1 \ge \frac{1}{2} \left( \|\vec{x}_1 - \vec{x}_2\|_1 - \|\vec{z}_1 - \vec{x}_1\|_1 - \|\vec{z}_2 - \vec{x}_2\|_1 \right) \ge \frac{1}{2}(10 - 2 - 2) = 3>2.$$

We may now assume that $\vec{z}_1, \vec{z}_2 \in \mathcal{R}_{\vec{x}}$ for some $\vec{x} \in \mathcal{P}$. Let $\vec{z}_1 = \vec{x} - \vec{e}_{\beta} + \vec{e}_{\gamma}$ and $\vec{z}_2 = \vec{x} - \vec{e}_{\beta'} + \vec{e}_{\gamma'}$. By the definition of $\mathcal{R}_{\vec{x}}$, $\|\vec{z}_1 - \vec{z}_2\|_1 = 4$, which implies $d_{G_{\mathcal{T}}}(\vec{z}_1, \vec{z}_2) \ge 2$. It remains to show that the case $d_{G_{\mathcal{T}}}(\vec{z}_1, \vec{z}_2) = 2$ is impossible. To show this, suppose that equality holds.  This implies the existence of an intermediate vertex $\vec{w} \in \mathcal{T} \cap \mathcal{B}_d^k$ satisfying $\|\vec{w} - \vec{z}_1\|_1 = 2$ and $\|\vec{w} - \vec{z}_2\|_1 = 2$. This forces $\vec{w} - \vec{x}$ to be one of the following: 
\[
\vec{0}, \quad -\vec{e}_{\beta} + \vec{e}_{\gamma'}, \quad -\vec{e}_{\beta'} + \vec{e}_{\gamma}, \quad \text{or} \quad -\vec{e}_{\beta} + \vec{e}_{\gamma} - \vec{e}_{\beta'} + \vec{e}_{\gamma'}.
\]

Since $\vec{w}, \vec{x} \in \mathcal{B}_d^k$, their defining congruences imply a constraint on their difference:
\[ \sum_{i=1}^d i(w_i - x_i) \equiv \sum_{i=1}^d i w_i - \sum_{i=1}^d i x_i \equiv 1 - 1 \equiv 0 \pmod d. \]

This congruence imposes strict conditions on the four possible values for $\vec{w}-\vec{x}$. If $\vec{w}-\vec{x}$ takes the value $-\vec{e}_{\beta} + \vec{e}_{\gamma'}$ or $-\vec{e}_{\beta'} + \vec{e}_{\gamma}$, the congruence respectively forces $\beta = \gamma'$ or $\beta' = \gamma$. Both scenarios thus reduce to the first possibility, $\vec{w}-\vec{x}=\vec{0}$. This, in turn, implies $\vec{w}=\vec{x} \in \mathcal{P}$, which contradicts $\vec{w} \in \mathcal{T}$ as $\mathcal{T}\cap\mathcal{P}=\emptyset$.

Finally, consider the case $\vec{w}-\vec{x} = -\vec{e}_{\beta} + \vec{e}_{\gamma} - \vec{e}_{\beta'} + \vec{e}_{\gamma'}$. For $d=3$, the set $\mathcal{R}_{\vec{x}}$ consists of at most two elements: $\vec{x} + \vec{e}_{\alpha+1}-\vec{e}_{\alpha-1}$ and $\vec{x} - \vec{e}_{\alpha+1}+\vec{e}_{\alpha-1}$. The distinct vertices $\vec{z}_1, \vec{z}_2$ must be these elements, yielding $\vec{w} = \vec{z}_1+\vec{z}_2-\vec{x} = \vec{x}$ and thus reducing to the case $\vec{w}-\vec{x}=\vec{0}$. For $d>3$, the constraint on $\vec{w}-\vec{x}$ requires that $(\gamma-\beta) + (\gamma'-\beta') \equiv 0 \pmod d$. By construction of $\mathcal{R}_{\vec{x}}$, the index differences $(\gamma-\beta)$ and $(\gamma'-\beta')$ must be $-1$ or $-2$ modulo $d$. 
The difference $-2$ arises uniquely from the element $\vec{x} - \vec{e}_{\alpha+1} + \vec{e}_{\alpha-1}$. Since $\vec{z}_1$ and $\vec{z}_2$ are distinct, their index differences cannot both be $-2$. Their sum is therefore congruent to $-2$ or $-3$, which is never $0 \pmod d$ as $d>3$.

As the existence of any intermediate vertex $\vec{w}$ leads to a contradiction,  the case $d_{G_{\mathcal{T}}}(\vec{z}_1, \vec{z}_2)=2$ is ruled out. We must therefore have $d_{G_{\mathcal{T}}}(\vec{z}_1, \vec{z}_2) > 2$, which provides the final contradiction.
\end{poc}

The following claim verifies that the family $\mathcal{T}$ satisfies the cycle prevention condition.

\begin{claim}\label{claim:cycle_prevention_restored}
For each connected component $C$ of $G_{\mathcal{T}}$, there exists an index set $I_C \subseteq [d]$ and an integer $v_C$ such that $q \nmid v_C$ and for all $\vec{w} \in C$, $\sum_{i \in I_C} w_i = v_C$.
\end{claim}

\begin{poc}
We first address the case where $C$ is disjoint from $\bigcup_{\vec{v} \in \mathcal{P}} \mathcal{R}_{\vec{v}}$. In this situation, $C \subseteq \mathcal{B}_d^k \setminus \mathcal{P}$. By Claim~\ref{claim:base_family_props}, $C$ must be a singleton, say $C = \{\vec{w}\}$. Since $\vec{w} \notin \mathcal{P}$, there exists an index $j \in [d]$ such that $q \nmid w_j$. Choosing $I_C = \{j\}$ and $v_C = w_j$ satisfies the required condition.

In the remaining case, $C$ must intersect with $\bigcup_{\vec{v} \in \mathcal{P}} \mathcal{R}_{\vec{v}}$. By Claim~\ref{claim:component_separation}, this intersection consists of exactly one type. Let this unique type be $\vec{z}$, where $\vec{z} \in \mathcal{R}_{\vec{x}}$ for some $\vec{x} \in \mathcal{P}$.

Any other type $\vec{w} \in C$ must belong to $\mathcal{B}_d^k \setminus \mathcal{P}$. By Claim~\ref{claim:base_family_props}, the types in $\mathcal{B}_d^k$ are isolated from one another, which means any such $\vec{w}$ must be adjacent to $\vec{z}$. This implies that $C$ has a star-like structure, consisting of the central type $\vec{z}$ and the set of its neighbors that lie in $\mathcal{B}_d^k \setminus \mathcal{P}$.

By Claim~\ref{claim:zero_component}, there is an index $\alpha$ such that $x_\alpha=0$. We identify these potential neighbors $\vec{w}\in \mathcal{B}_d^k \setminus \mathcal{P}$ using the congruence conditions in the definition of $\mathcal{B}_d^k$.

\begin{itemize}
\item If $\vec{z} = \vec{x} - \vec{e}_i + \vec{e}_{i-1}$ for some $i \in [d] \setminus \{\alpha, \alpha+1\}$, any neighbor of $\vec{z}$ must have the form $\vec{w} = \vec{z} - \vec{e}_m + \vec{e}_{m+1}$ for some $m$ since $\vec{w}\in\mathcal{B}_d^k\setminus \mathcal{P}$. Notice that  $m\neq i-1$ since otherwise $\vec{w}= \vec{x}\notin \mathcal{B}_d^k \setminus \mathcal{P}$, and $m\neq \alpha$ since otherwise $\alpha$-coordinate of $\vec{w}$ is negative. We set $I_C = \{i, i+1, i+2, \dots, \alpha\}$ and $v_C = \sum_{k \in I_C} z_k= \sum_{k \in I_C} x_k-1$.

\item If $\vec{z} = \vec{x} - \vec{e}_{\alpha+1} + \vec{e}_{\alpha-1}$, any neighbor of $\vec{z}$ must have the form  $\vec{w} = \vec{z} - \vec{e}_m + \vec{e}_{m+2}$ for some $m$ since $\vec{w}\in\mathcal{B}_d^k\setminus \mathcal{P}$. Notice that  $m\neq \alpha-1$ since otherwise $\vec{w}= \vec{x}\notin \mathcal{B}_d^k \setminus \mathcal{P}$, and $m\neq \alpha$ since otherwise $\alpha$-coordinate of $\vec{w}$ is negative. As $d$ is odd, we set $I_C = \{\alpha-1+2i: i\in [\frac{d+1}{2}]\}=\{\alpha+1, \alpha+3,  \dots, \alpha-2, \alpha\}$ and $v_C = \sum_{k \in I_C} z_k= \sum_{k \in I_C} x_k-1$.
\end{itemize}

It is straightforward to verify that for any $(w_1, \dots, w_d) \in C$, we have $\sum_{i \in I_C} w_i = v_C$, since $C$ consists only of the central type $\vec{z}$ and its potential neighbors whose forms were determined in the cases above.
Furthermore, as $\vec{x} \in \mathcal{P}$, each $x_i$ is a multiple of $q$. Thus $v_C = \sum_{k \in I_C} x_k-1\equiv -1  \not\equiv 0 \pmod q$. This establishes that the index set $I_C$ and the integer $v_C$ satisfy the required conditions.\end{poc}

In summary, the constructed family $\mathcal{T} = (\mathcal{B}_d^k \setminus \mathcal{P}) \cup (\bigcup_{\vec{x} \in \mathcal{P}} \mathcal{R}_{\vec{x}})$ is a $(k,\ell; d)$-family as defined in Definition \ref{def:kld_family}, since the first  property is fulfilled by Claim~\ref{claim:base_family_props} and Claim~\ref{claim:extension_restored}, and the second property is fulfilled by Claim~\ref{claim:cycle_prevention_restored}.
This completes the proof of \cref{thm:existence}.
\end{proof}

\section{The upper bound: reducing \cref{thm:main_upper_bound} to \cref{thm:structural}}\label{sec:4}
This section is devoted to proving \cref{thm:main_upper_bound}. We will first derive \cref{thm:main_upper_bound} assuming \cref{thm:structural}, and defer the technical proof of \cref{thm:structural} to \cref{sec:proof_of_structural}.

\subsection{Preliminaries}

To proceed, we first introduce some definitions. Let $G$ be a  $k$-uniform hypergraph. An \emph{ordered edge} is a tuple $\vec{e}=(v_1, v_2, \ldots, v_k)$ where $e = \{v_1, v_2, \ldots, v_k\} \in E(G)$, where different vertex orderings of the same edge are considered distinct. For a permutation $\sigma$ on $[k]$, its action on an ordered edge $\vec{e}$ is $\sigma(\vec{e}) = (v_{\sigma(1)}, v_{\sigma(2)}, \ldots, v_{\sigma(k)})$. A \emph{walk} of length $t\geq 0$ (which contains $t+1$ edges) is a sequence of vertices $v_1v_2\dots v_{t+k}$ such that for every $0 \leq i \leq t$, the set $\{v_{i+1}, v_{i+2}, \ldots, v_{i+k}\}$ is an edge. We often refer to this as a walk from the edge $\{v_1, \dots, v_k\}$ to $\{v_{t+1}, \dots, v_{t+k}\}$, or more precisely, from the ordered edge $(v_1, \dots, v_k)$ to $(v_{t+1}, \dots, v_{t+k})$, depending on whether the context requires such precision.

Walks have two simple but crucial properties that we state as facts.
\begin{fact}
\label{fact:walk_concatenation}
Given a walk of length $t_1$ from ordered edge $\vec{e}_1$ to $\vec{e}_2$ and a walk of length $t_2$ from $\vec{e}_2$ to $\vec{e}_3$, these walks can be concatenated to form a walk of length $t_1 + t_2$ from $\vec{e}_1$ to $\vec{e}_3$.
\end{fact}

\begin{fact}
\label{fact:trivial_walk}
For any ordered edge $\vec{e} = (v_1,v_2,\dots,v_k)$, there exists a walk of length $k$ from $\vec{e}$ to itself, namely $v_1v_2\dots v_{k-1}v_k\ v_1v_2\dots v_{k-1}v_k$.
\end{fact}

To understand the local structure around a given edge and how it connects to its neighborhood, we introduce the following key concept of exchangeability. This allow us to construct specific walks that reorder vertices within an edge.

\begin{definition}
\label{def:exchangeable}
Given an edge $e\in E(G)$, two vertices $u,v \in e$ are \emph{exchangeable in $e$} if there exists a vertex $w \notin e$ such that both $(e \setminus \{u\}) \cup \{w\}$ and $(e \setminus \{v\}) \cup \{w\}$ are edges in $G$.
\end{definition}

The following two auxiliary lemmas are crucial in our proofs. They establish the connection between the local property of exchangeability and the existence of specific walks needed to find the forbidden cycle.
\begin{lemma}
\label{lemma:forming_good_walk}
Let $\sigma$ be any permutation on $[k]$ and $\vec{e} = (v_1,v_2,\dots,v_k)$ be an ordered edge. If vertices $v_i$ and $v_j$ are exchangeable in $e$, then there exists a walk of length $2k$ from $\sigma(\vec{e})$ to $((i\;j) \circ \sigma)(\vec{e})$.
\end{lemma}

\begin{proof}
Since $v_i$ and $v_j$ are exchangeable in $e$, there exists a vertex $w$ such that both $(e \setminus \{v_i\}) \cup \{w\}$ and $(e \setminus \{v_j\}) \cup \{w\}$ are edges in $G$.

Let $a,b$ be such that $\sigma(a)=i$ and $\sigma(b)=j$. Without loss of generality, assume $a<b$. Then
\[\begin{aligned}
    \sigma(\vec{e}) &= (v_{\sigma(1)}, \dots, v_{\sigma(a-1)}, v_i, v_{\sigma(a+1)}, \dots, v_{\sigma(b-1)}, v_j, v_{\sigma(b+1)}, \dots, v_{\sigma(k)}),\\
    ((i\;j) \circ \sigma)(\vec{e})&= (v_{\sigma(1)}, \dots, v_{\sigma(a-1)}, v_j, v_{\sigma(a+1)}, \dots, v_{\sigma(b-1)}, v_i, v_{\sigma(b+1)}, \dots, v_{\sigma(k)}).
\end{aligned}\]
The following sequence forms a valid walk of length $2k$ from $\sigma(\vec{e})$ to $((i\;j) \circ \sigma)(\vec{e})$, as any consecutive $k$ vertices form either $e$, $(e \setminus \{v_i\}) \cup \{w\}$, or $(e \setminus \{v_j\}) \cup \{w\}$, all of which are edges in $G$:
$$v_{\sigma(1)}\dots v_i\dots v_j\dots v_{\sigma(k)}\ \ \  v_{\sigma(1)}\dots w\dots v_i\dots v_{\sigma(k)}\ \ \ v_{\sigma(1)}\dots v_j\dots v_i\dots v_{\sigma(k)}.$$
This completes the proof of this lemma.
\end{proof}

The second lemma provides a sufficient condition for finding a \emph{homomorphic copy} of $C^k_\ell$. In this context, a map $\phi: V(F) \to V(G)$ is called a \emph{homomorphism} if it maps every edge $\{v_1, \dots, v_k\} \in E(F)$ to an edge $\{\phi(v_1), \dots, \phi(v_k)\} \in E(G)$. If such a map exists, we say that $G$ contains a homomorphic copy of $F$.

\begin{lemma}
\label{lemma:cycle_existence}
Let $3\le k< \ell$ be integers, $G$ be a $k$-uniform hypergraph, and $t \in \{0, 1, \ldots, k-1\}$ satisfy $\ell \equiv t \pmod{k}$.  If there exists a walk $W$ in $G$ from an ordered edge $\vec{e} = (v_1, \ldots, v_k)$ to its cyclic shift $(1 \, 2 \, \ldots \, k)^{k-t}(\vec{e}) = (v_{k-t+1}, \ldots, v_k, v_1, \ldots, v_{k-t})$ whose length is divisible by $k$ and at most $\ell-k$, then $G$ contains a homomorphic copy of $C^k_\ell$. 
\end{lemma}
\begin{proof}
By appending the $t$ vertices $v_{k-t+1}, \ldots, v_k$ to the end of $W$, we obtain a new walk $W'$ from $\vec{e}$ to $\vec{e}$. The length of $W'$ is congruent to $t \pmod{k}$ and is at most $\ell - k + t < \ell$.
Since both $\ell$ and the length of $W'$ are congruent to $t \pmod{k}$, we can extend $W'$ by an appropriate number of $k$-length walks from $\vec{e}$ to itself (which exist by Fact~\ref{fact:trivial_walk}) to obtain a walk from $\vec{e}$ to $\vec{e}$ of length exactly $\ell$. This walk forms a homomorphic copy of $C^k_\ell$.
\end{proof}

\subsection{Deriving Theorem~\ref{thm:main_upper_bound} from Theorem~\ref{thm:structural}}

In this subsection, we demonstrate how Theorem~\ref{thm:structural} implies our main result, Theorem~\ref{thm:main_upper_bound}. The core of the argument is to show that the structure guaranteed by Theorem~\ref{thm:structural} leads to a contradiction when combined with the given degree conditions.

\begin{proof}[Proof of Theorem~\ref{thm:main_upper_bound}]
Our proof proceeds by contradiction. We first note that by the phenomenon of supersaturation (see, e.g.~\cite[Section 2]{K2011} and~\cite[Proposition 1.4]{MZ2007}), to establish an upper bound for $\gamma(C^k_\ell)$, it suffices to find a homomorphic image of $C^k_\ell$, rather than a genuine copy.
We therefore suppose for the sake of contradiction that there exists a  $k$-uniform hypergraph $G$ with $\delta_{k-1}(G) > \frac{1}{3}|V(G)|$ that contains no homomorphic copy of $C^k_\ell$, where $k \nmid \ell$, $\ell \geqslant 20k^2$, and $2 \nmid \frac{k}{\gcd(k,\ell)}$. By Theorem~\ref{thm:structural}, there exists a vertex subset $B \subseteq V(G)$ with $\frac{|V(G)|}{3} \leqslant |B| \leqslant \frac{2|V(G)|}{3}$ such that $|e \cap B|$ is odd for every edge $e \in E(G)$.

Let $s = \frac{k}{\gcd(k,\ell)}$. By assumption, $s\geq 3$ is odd. Consider any $(k-1)$-set $S$ containing exactly $s-1$ vertices from $B$ and $k-s$ vertices from $B^c$. Since $\delta_{k-1}(G) > \frac{1}{3}|V(G)|$, the common neighborhood $N(S)$ is non-empty.

For any $v \in N(S)$, the edge $e = S \cup \{v\}$ must satisfy $|e \cap B|$ is odd. If $v \in B^c$, then $|e \cap B| = s-1$, which is even since $s$ is odd. This contradicts the property of $B$. Therefore, $N(S) \subseteq B$, and we can find an edge $e$ with exactly $s$ vertices in $B$ and $k-s$ vertices in $B^c$.

\begin{claim}
For any edge $e$ with $|e \cap B| = s$, the following hold:
\begin{enumerate}
\item Any two vertices in $e \cap B$ are exchangeable in $e$.
\item Any two vertices in $e \cap B^c$ are exchangeable in $e$.
\end{enumerate}
\end{claim}

\begin{proof}[Proof of Claim]
Consider two vertices $u,v \in e \cap B$. Let $S_u = e \setminus \{u\}$ and $S_v = e \setminus \{v\}$. Note that $|S_u \cap B| = |S_v \cap B| = s-1$, which is even. For any edge containing $S_u$, it must have an odd intersection with $B$. Since $|S_u \cap B|$ is even, any vertex in $N(S_u)$ must belong to $B$. Similarly, $N(S_v) \subseteq B$. 

By the degree condition, both $|N(S_u)|$ and $|N(S_v)|$ exceed $\frac{1}{3}|V(G)|$. Since both neighborhoods are subsets of $B$ and $\frac{|V(G)|}{3} \leqslant |B| \leqslant \frac{2|V(G)|}{3}$, we have $N(S_u) \cap N(S_v) \neq \emptyset$. For $w\in N(S_u) \cap N(S_v) $, $(e \setminus \{u\}) \cup \{w\}$ and $(e \setminus \{v\}) \cup \{w\}$ are both edges in $G$, showing that $u,v$ are exchangeable in $e$.

The proof for vertices in $e \cap B^c$ is analogous.
\end{proof}

We now construct the required walk to apply Lemma~\ref{lemma:cycle_existence}.
Let $t \in \{0, 1, \ldots, k-1\}$ satisfy  $\ell \equiv t \pmod{k}$ and let $g = \gcd(k,\ell)= \gcd(k,t)$, then $s = \frac{k}{\gcd(k,\ell)}= \frac{k}{g}$. The cyclic permutation $(1 \, 2 \, \ldots \, k)^{k-t}$ decomposes into exactly $g$ disjoint orbits of length $s$:
$$(1 \, 2 \, \ldots \, k)^{k-t} =  \pi_1 \circ \pi_2\circ \cdots \circ \pi_{g}$$
where
$$
\pi_i = (i, \, i+(k-t), \, i+2(k-t), \, \ldots, \, i+(s-1)(k-t))
$$
for $i = 1,2, \ldots, g$, with all arithmetic modulo $k$.

For our edge $e$ with $|e \cap B| = s$, let $\vec{e} = (v_1, v_2, \ldots, v_k)$ be an ordered edge where $v_j \in B$ if and only if $j$ is a multiple of $g$. Then $\pi_g$ contains only vertices from $e \cap B$, while all other orbits $\pi_i$ contain only vertices from $e \cap B^c$.

By our claim, the vertices within $e \cap B$ are pairwise exchangeable, as are the vertices within $e \cap B^c$. Then each orbit $\pi_i$ can be realized as a composition of at most $s$ transpositions of exchangeable vertices. Since we have $g$ orbits, the permutation $(1 \, 2 \, \ldots \, k)^{k-t}$ can be decomposed into at most $g \cdot s = k$ transpositions of exchangeable vertices. 

For each such transposition of exchangeable vertices, Lemma~\ref{lemma:forming_good_walk} gives us a walk of length $2k$. By Fact~\ref{fact:walk_concatenation}, these walks can be concatenated to form a walk from $\vec{e}$ to $(1 \, 2 \, \ldots \, k)^{k-t}(\vec{e})$ with total length divisible by $k$ and at most $k \cdot 2k = 2k^2 < \ell - k$.
By Lemma~\ref{lemma:cycle_existence}, $G$ contains a homomorphic copy of $C^k_\ell$, contradicting our assumption.\end{proof}

\section{Proof of \cref{thm:structural}}
\label{sec:proof_of_structural}

In this section, we prove \cref{thm:structural}, which provides a structural characterization of dense hypergraphs without a homomorphic copy of $C^k_\ell$.
The proof of Theorem~\ref{thm:structural} proceeds by establishing a connection between a local structure within each edge and a global partition of the vertex set. This section is organized as follows: we first introduce the necessary concepts, then present an outline of the proof, and finally provide the detailed arguments for each step.

Throughout this section, we fix $G$ to be a $k$-uniform hypergraph with $\delta_{k-1}(G) > \frac{1}{3} |V(G)|$ that contains no homomorphic copy of $C^k_{\ell}$, where $\ell$ is a fixed integer satisfying $\ell \geq 20k^2$ and $k \nmid \ell$. We begin by analyzing the local structure of an edge $e \in E(G)$ through the exchangeability of its vertices. Based on the lack of exchangeability, we define an auxiliary graph $H_e$ as follows: 
\begin{definition}
\label{def:auxiliary_graph}
Given an edge $e$, the auxiliary graph $H_e$ has the vertex set of $e$. Two vertices are adjacent in $H_e$ if and only if they are not exchangeable in $e$. In its complement, $\overline{H}_e$, an edge signifies that two vertices are exchangeable.
\end{definition}

The second concept provides the language for describing a global partition of the vertex set. A central part of our argument will be to show that the local structure of any given edge is uniquely determined by such a global partition. We will ultimately prove that this determining partition is, in fact, the same for all edges in the graph.

\begin{definition}
\label{def:slice}
Let $B \subseteq V(G)$ be a vertex set with $\frac{|V(G)|}{3} \leq |B| \leq \frac{2|V(G)|}{3}$, and let $1\leq i \leq k-1$ be an integer. We define $\mathrm{Slice}(B,i)$ to be the collection of all $k$-subsets of $V(G)$ that have exactly $i$ vertices in $B$. Note that $\mathrm{Slice}(B,i)$ and $\mathrm{Slice}(V(G) \setminus B, k-i)$ are considered the same slice.
\end{definition}

To extend the structural properties from a single edge to the entire hypergraph, several concepts relating different edges are required. We say two edges are \emph{adjacent} if they share $k-1$ vertices. Two edges are \emph{tightly connected} if there exists a sequence of edges starting with one and ending with the other, where any two consecutive edges in the sequence are adjacent. This defines an equivalence relation on the edge set, partitioning it into \emph{tight components}. The tight component containing an edge $e$ is denoted by $\mathcal{C}(e)$. The detailed proof also employs two standard definitions. The \emph{star} of a vertex set $A$, denoted $S(A)$, is the set of all edges in $G$ containing $A$. The \emph{link graph}, $L(A)$, is the hypergraph with vertex set $V(G) \setminus A$ and edge set $\{e \setminus A \mid e \in S(A)\}$. It is thus a $(k-|A|)$-uniform hypergraph formed by removing $A$ from each edge in the star of $A$.

With these concepts established, the proof of \cref{thm:structural} will be divided into the four steps outlined below.
\newlength{\steplabelwidth}
\settowidth{\steplabelwidth}{\underline{Step 4:}\quad} 
\begin{enumerate}[
    label=\underline{Step \arabic*:},
    leftmargin=\steplabelwidth,
    labelwidth=\steplabelwidth,
    labelsep=0pt,
    itemindent=0pt,
    align=left
]
    \item We first show that for any edge $e$, its auxiliary graph $H_e$ is a complete bipartite graph. Furthermore, this local bipartition structure is preserved between adjacent edges.
    \item We then show that the local bipartition within an edge $e$ is induced by a unique global slice that contains not only $e$ but also all edges in the star of any $(k-2)$-subset of $e$ formed by removing a non-exchangeable pair.
    \item We establish that all edges in an arbitrary tight component $\mathcal{C}(e)$ are contained in a common unique slice. Building on this, we show that this slice provides a complete characterization of $\mathcal{C}(e)$.
    \item Finally, we demonstrate that the slices corresponding to different tight components are determined by the same vertex set $B$, and that for every edge $e$, the intersection sizes $|e \cap B|$ and $|e\cap (V(G)\setminus B)|$ are odd integers.
\end{enumerate}

We will provide the detailed arguments for each of these steps in the subsections that follow.

\subsection{Step 1: Establishing the Local Bipartite Structure of $H_e$}
\label{subsec:step1}

In this section, we analyze the local structure induced by the vertex exchangeability relation. Our first lemma specifies this structure precisely, stating that for any given edge $e$, the auxiliary graph $H_e$ formed by its non-exchangeable vertex pairs must be a complete bipartite graph.

\begin{lemma}
\label{lemma:edge_partition}
Let $G$ be a $k$-uniform hypergraph with $\delta_{k-1}(G) > \frac{1}{3} |V(G)|$ that contains no homomorphic copy of $C^k_{\ell}$, where $\ell$ is a fixed integer satisfying $\ell \geq 20k^2$ and $k \nmid \ell$. Then for any edge $e \in E(G)$, its auxiliary graph $H_e$ is a complete bipartite graph.
\end{lemma}

\begin{proof}[Proof of Lemma~\ref{lemma:edge_partition}]

Our proof proceeds in two parts. First, we show that if the complement graph $\overline{H}_e$ is disconnected, then $H_e$ is a complete bipartite graph. Second, we prove that $\overline{H}_e$ must be disconnected.

We first prove that $\overline{H}_e$ has no independent set of size 3. To prove this, consider any three distinct vertices $v_1,v_2,v_3 \in e$ and their corresponding sets $W_i = \{w\in V(G) \colon (e\setminus\{v_i\})\cup\{w\} \in E(G)\}$. The minimum codegree condition implies $|W_i| > |V(G)|/3$ for each $i \in [3]$. As the sum of their sizes exceeds $|V(G)|$, the pigeonhole principle guarantees that at least two of these sets, say $W_i$ and $W_j$, must intersect. This implies that vertices $v_i$ and $v_j$ are exchangeable, and thus form an edge in $\overline{H}_e$.

We now show this property implies that if $\overline{H}_e$ is disconnected, then $H_e$ is a complete bipartite graph. If $\overline{H}_e$ is disconnected, let $C$ be one of its connected components. The property that $\overline{H}_e$ has no independent set of size 3 ensures that both $C$ and its complement $e \setminus C$ must be cliques in $\overline{H}_e$. Consequently, in the original graph $H_e$, there are no edges within $C$ or $e \setminus C$, and all edges run between them. Thus, $H_e$ is a complete bipartite graph with parts $(C, e \setminus C)$, which proves the first part of our strategy.

It now suffices to show that $\overline{H}_e$ must be disconnected.
We proceed by contradiction, assuming that $\overline{H}_e$ is connected. Under this assumption, the graph must have a small diameter. Specifically, the distance between any two vertices in $\overline{H}_e$ is at most 3. To see this, suppose for contradiction that there exist vertices $x, y \in e$ with distance at least 4, and let $v_0, v_1, \ldots, v_t$ be a shortest path from $x=v_0$ to $y=v_t$ with $t \geq 4$. The set $\{v_0, v_2, v_t\}$ must contain an edge in $\overline{H}_e$, and any such edge would create a shorter path from $x$ to $y$, a contradiction.

The crucial step is to show that this bounded diameter allows us to construct specific walks, which is formalized in the following claim.

\begin{claim}\label{lemma:permutation_walk}
If $\overline{H}_e$ is connected, then for any permutation $\sigma$ on $[k]$ and any ordered edge $\vec{e}$, there exists a walk from $\vec{e}$ to $\sigma(\vec{e})$ whose length is divisible by $k$ and at most $10k^2-10k$.
\end{claim}
\begin{poc}
We first decompose an arbitrary permutation $\sigma$ into a sequence of at most $k-1$ transpositions $\sigma = \tau_m \circ \tau_{m-1} \circ \cdots \circ \tau_1$ where $m \leq k-1$. Let $\sigma_0 = \mathrm{id}$ and define $\sigma_i = \tau_i \circ \sigma_{i-1}$ for $1 \leq i \leq m$, where $\tau_i$ are the transpositions. Our goal is to construct a walk for each $\tau_i$ that takes us from $\sigma_{i-1}(\vec{e})$ to $\sigma_i(\vec{e})$.

For a given transposition $\tau_i = (a, b)$, we know that vertices $v_a$ and $v_b$ are connected by a path of length at most 3 in $\overline{H}_e$. This allows us to express $\tau_i$ as a composition of at most 5 transpositions between exchangeable pairs. For instance, for a distance-3 path $v_a-v_x-v_y-v_b$ where consecutive vertices are exchangeable in $e$, the transposition $\tau_i$ can be written as $(a, b)=(a,x)\circ(x,y)\circ(y,b)\circ(x,y)\circ(a,x)$.

Applying Lemma~\ref{lemma:forming_good_walk} to each elementary transposition in this decomposition yields a sequence of elementary walks, each of length $2k$. Concatenating these walks constructs the desired walk from $\sigma_{i-1}(\vec{e})$ to $\sigma_i(\vec{e})$. Since this involves at most 5 elementary walks, this resulting walk has a length divisible by $k$ and at most $10k$.

By concatenating these walks from $\sigma_{i-1}(\vec{e})$ to $\sigma_i(\vec{e})$ for each of the $m$ transpositions, we construct a single walk from $\sigma_0(\vec{e})=\vec{e}$ to $\sigma_m(\vec{e})=\sigma(\vec{e})$. The total length of this walk is divisible by $k$ and is at most $m \times 10k \leq (k-1) \times 10k = 10k^2 - 10k$.\end{poc}

With Claim~\ref{lemma:permutation_walk} established, the final contradiction follows. Let $t\in\{0,\dots,k-1\}$ be such that $\ell \equiv t \pmod{k}$, and let $\sigma$ be the cyclic shift $(1 \, 2 \, \ldots \, k)^{k-t}$. By Claim~\ref{lemma:permutation_walk}, there exists a walk from $\vec{e}$ to $\sigma(\vec{e})$ whose length is divisible by $k$ and at most $10k^2 - 10k<\ell-k$. By Lemma~\ref{lemma:cycle_existence}, this implies that $G$ contains a homomorphic copy of $C^k_\ell$, which contradicts the assumption on $G$. Therefore, the premise that $\overline{H}_e$ is connected must be false. This completes the proof.\end{proof}

Having established that the non-exchangeable relation in each edge forms a local bipartite structure, the following lemma shows that this structure is consistent across adjacent edges.

\begin{lemma}
\label{lemma:auxiliary_isomorphism}
Let $G$ be a $k$-uniform hypergraph with $\delta_{k-1}(G) > \frac{1}{3} |V(G)|$ that contains no homomorphic copy of $C^k_{\ell}$, where $\ell$ is a fixed integer satisfying $\ell \geq 20k^2$ and $k \nmid \ell$. Then for any two adjacent edges $e_1=v_1v_2\dots v_k$ and $e_2=v_2v_3\dots v_{k+1}$, their auxiliary graphs $H_{e_1}$ and $H_{e_2}$ are isomorphic via the mapping that sends $v_1$ to $v_{k+1}$ and preserves all other vertices.
\end{lemma}

\begin{proof}[Proof of Lemma~\ref{lemma:auxiliary_isomorphism}]
We prove this by contradiction. By Lemma~\ref{lemma:edge_partition}, both $H_{e_1}$ and $H_{e_2}$ are complete bipartite graphs. Let $\vec{e}_1 = (v_1, v_2, \dots, v_k)$ and $\vec{e}_2 = (u_1, u_2, \dots, u_k)$ be ordered representations of $e_1$ and $e_2$, where $u_1 = v_{k+1}$ and $u_i = v_i$ for $2 \leq i \leq k$. Assume, for the sake of contradiction, that $H_{e_1}$ and $H_{e_2}$ are not isomorphic under the correspondence $v_i \leftrightarrow u_i$ for $i \in [k]$. Then there exist indices $a, b$ with $1 \leq a < b \leq k$ such that $\{v_a, v_b\} \in E(H_{e_1})$ but $\{u_a, u_b\} \notin E(H_{e_2})$. This implies that the pair $\{v_a, v_b\}$ is not exchangeable in $e_1$, while the corresponding pair $\{u_a, u_b\}$ is exchangeable in $e_2$. 

To leverage this difference, we first formalize how walks can be constructed between $e_1$ and $e_2$:

\begin{fact}\label{fact:walks_between_consecutive_edges}
For any permutation $\sigma$ on $[k]$:
\begin{itemize}
    \item $v_{\sigma(1)}v_{\sigma(2)}\ldots v_{\sigma(k)}u_{\sigma(1)}u_{\sigma(2)}\ldots u_{\sigma(k)}$ is a walk of length $k$ from $\sigma(\vec{e}_1)$ to $\sigma(\vec{e}_2)$.
    \item $u_{\sigma(1)}u_{\sigma(2)}\ldots u_{\sigma(k)}v_{\sigma(1)}v_{\sigma(2)}\ldots v_{\sigma(k)}$ is a walk of length $k$ from $\sigma(\vec{e}_2)$ to $\sigma(\vec{e}_1)$.
\end{itemize}
Since any consecutive $k$ vertices in each sequence form either $e_1$ or $e_2$, these are valid walks.
\end{fact}
As the pair $\{v_a, v_b\}$ is not exchangeable in $e_1$, we cannot directly apply Lemma~\ref{lemma:forming_good_walk} to generate a walk that swaps them. The following claim, however, demonstrates how routing through $e_2$ provides an effective substitute to achieve this exact outcome.
\begin{claim}\label{lemma:permutation_walk2}
There exists a walk of length $4k$ from $\sigma(\vec{e_1})$ to $((a\;b) \circ \sigma)(\vec{e_1})$ for any permutation $\sigma$.
\end{claim}
\begin{poc}
Since $u_a$ and $u_b$ are exchangeable in $e_2$, by Lemma~\ref{lemma:forming_good_walk}, there exists a walk of length $2k$ from $\sigma(\vec{e_2})$ to $((a\;b) \circ \sigma)(\vec{e_2})$.
By Fact~\ref{fact:walks_between_consecutive_edges}, there is a walk of length $k$ from $\sigma(\vec{e_1})$ to $\sigma(\vec{e_2})$, and a walk of length $k$ from $((a\;b) \circ \sigma)(\vec{e_2})$ to $((a\;b) \circ \sigma)(\vec{e_1})$.
Concatenating these walks, we obtain a walk of length $4k$ from $\sigma(\vec{e_1})$ to $((a\;b) \circ \sigma)(\vec{e_1})$:
\[
\sigma(\vec{e_1}) \xrightarrow{k} \sigma(\vec{e_2}) \xrightarrow{2k} ((a\;b) \circ \sigma)(\vec{e_2}) \xrightarrow{k} ((a\;b) \circ \sigma)(\vec{e_1}). \qedhere
\]
\end{poc}
The preceding claim provides the crucial mechanism for handling the non-exchangeable pair $\{v_a, v_b\}$. The following claim now uses this to construct a walk for an arbitrary permutation, which will lead to the final contradiction. The argument mirrors the one in the proof of Claim~\ref{lemma:permutation_walk}, but utilizes the preceding claim to handle any transpositions involving the non-exchangeable pair.

\begin{claim}\label{lemma:permutation_walk3}
For any permutation $\sigma$ on $[k]$, there exists a walk from $\vec{e_1}$ to $\sigma(\vec{e_1})$ whose length is divisible by $k$ and at most $20k^2-20k$.
\end{claim}

\begin{proof}
We decompose an arbitrary permutation $\sigma$ into a sequence of $m \leq k-1$ transpositions, $\sigma = \tau_m \circ \cdots \circ \tau_1$. Let $\sigma_0 = \mathrm{id}$ and $\sigma_i = \tau_i \circ \sigma_{i-1}$.

As $\overline{H}_{e_1}$ is a complete bipartite graph, for each transposition $\tau_i = (c,d)$, the vertices $v_c$ and $v_d$ are at most distance 3 apart in $\overline{H}_{e_1}$ with the edge $\{v_a, v_b\}$ added. This allows $\tau_i$ to be realized by at most 5 transpositions, each involving either an exchangeable pair from $e_1$ or the specific pair $\{v_a,v_b\}$. Applying Lemma~\ref{lemma:forming_good_walk} or Claim~\ref{lemma:permutation_walk2} to each of these transpositions yields a walk from $\sigma_{i-1}(\vec{e_1})$ to $\sigma_{i}(\vec{e_1})$ of length divisible by $k$ and at most $5 \cdot 4k = 20k$.

By concatenating these walks from $\sigma_{i-1}(\vec{e_1})$ to $\sigma_i(\vec{e_1})$ for each of the $m$ transpositions, we construct a single walk from $\sigma_0(\vec{e_1})=\vec{e_1}$ to $\sigma_m(\vec{e_1})=\sigma(\vec{e_1})$. The total length of this walk is divisible by $k$ and is at most $m \times 20k \leq (k-1) \times 20k = 20k^2 - 20k$.
\end{proof}

With Claim~\ref{lemma:permutation_walk3} established, the final contradiction follows. Let $t\in\{0,1,\dots,k-1\}$ be such that $\ell \equiv t \pmod{k}$, and let $\sigma$ be the cyclic shift $(1 \, 2 \, \ldots \, k)^{k-t}$. By Claim~\ref{lemma:permutation_walk3} there exists a walk from $\vec{e_1}$ to $\sigma(\vec{e_1})$ whose length is divisible by $k$ and at most $20k^2-20k\leq \ell-k$. By  Lemma~\ref{lemma:cycle_existence}, this implies that $G$ contains a homomorphic copy of $C^k_\ell$, yielding a contradiction. Thus, our initial assumption must be false, and $H_{e_1}$ and $H_{e_2}$ are isomorphic under the given correspondence. \end{proof}

\subsection{Step 2: Connecting Local Bipartitions to a Global Slice}
\label{subsec:step2}

In this section, we connect the local bipartite structure of an edge to a global partition of the vertex set. The main tool for this step is the link graph formed by removing a non-exchangeable pair of vertices from an edge. We begin by establishing the key properties of this link graph.

\begin{lemma}
\label{lemma:link_graph_properties}
Let $G$ be a $k$-uniform hypergraph with $\delta_{k-1}(G) > \frac{1}{3} |V(G)|$ that contains no homomorphic copy of $C^k_{\ell}$, where $\ell$ is a fixed integer satisfying $\ell \geq 20k^2$ and $k \nmid \ell$. For any edge $e \in E(G)$ and any pair of non-exchangeable vertices $\{u,v\} \in E(H_e)$, consider the link graph $L := L(e \setminus \{u,v\})$. Then $L$ is connected, bipartite, and has a minimum degree $\delta(L) > |V(G)|/3$.
\end{lemma}

\begin{proof}[Proof of Lemma~\ref{lemma:link_graph_properties}]
First, we establish the minimum degree. For any vertex $x \in V(L)$, the set $x \cup (e \setminus \{u,v\})$ is a $(k-1)$-subset of $V(G)$. By the minimum codegree condition on $G$, there are more than $|V(G)|/3$ vertices $z$ such that $(x \cup (e \setminus \{u,v\})) \cup \{z\}$ is an edge in $G$. Each such $z$ corresponds to a neighbor of $x$ in $L$, so $\deg_L(x) > |V(G)|/3$.

Next, we prove connectivity. The condition that $\{u,v\}$ is non-exchangeable in $e$ implies that their neighborhoods in the link graph are disjoint, $N_{L}(u) \cap N_{L}(v) = \emptyset$. For any vertex $x \in V(L)$, the degree sum $|N_{L}(x)|+|N_{L}(u)|+|N_{L}(v)| > |V(L)|$ forces $N_{L}(x)$ to have a non-empty intersection with $N_{L}(u) \cup N_{L}(v)$. This ensures any vertex $x$ has a path of length at most 2 to either $u$ or $v$. Since $\{u,v\}$ is also an edge in $L$, the entire graph is connected.

Finally, we show that $L$ is bipartite. Assume for contradiction that $L$ is not bipartite, so it must contain an odd cycle. As we have established that $L$ is connected, this implies the existence of a closed walk of odd length which contains the edge $\{u, v\}$. By applying Lemma~\ref{lemma:auxiliary_isomorphism} iteratively along the sequence of hyperedges corresponding to this walk, we induce an automorphism on the auxiliary graph $H_e$. The odd length of the walk forces this automorphism to swap $u$ and $v$. This is impossible, as $u$ and $v$ belong to different partitions in the complete bipartite graph $H_e$. Therefore, $L$ must be bipartite.
\end{proof}

With the structure of the link graph established, we can now prove that the local partition within an edge corresponds to a unique global slice.  The following lemma establishes this for the star of a fixed non-exchangeable pair.

\begin{lemma}
\label{lemma:unique_slice_star}
Let $G$ be a $k$-uniform hypergraph with $\delta_{k-1}(G) > \frac{1}{3} |V(G)|$ that contains no homomorphic copy of $C^k_{\ell}$, where $\ell$ is a fixed integer satisfying $\ell \geq 20k^2$ and $k \nmid \ell$. For any edge $e \in E(G)$ and any pair of non-exchangeable vertices $\{u,v\} \in E(H_e)$, there exists a unique slice $\mathrm{Slice}(B,i)$ such that all edges in the star $S(e \setminus \{u,v\})$ are contained in this slice and the two parts of the complete bipartite graph $H_e$ are precisely $e \cap B$ and $e \setminus B$.
\end{lemma}

\begin{proof}[Proof of Lemma~\ref{lemma:unique_slice_star}]
Let $L$ be the link graph $L(e \setminus \{u,v\})$. By Lemma~\ref{lemma:link_graph_properties}, $L$ is connected and bipartite. Let $(\mathcal{P}, \mathcal{Q})$ be the unique bipartition of its vertex set, $V(L)$, with $u \in \mathcal{P}$ and $v \in \mathcal{Q}$ without loss of generality. Let $(E_u, E_v)$ be the bipartition of $H_e$, with $u \in E_u$ and $v \in E_v$.

First, we prove existence by constructing the slice $\mathrm{Slice}(B, i)$ where we define the set $B = E_u \cup \mathcal{P}$ and $i = |E_u|$. We verify this construction. The parts of $H_e$ are indeed $e \cap B = E_u$ and $e \setminus B = E_v$. Now, consider any edge in the star $S(e \setminus \{u,v\})$. Such an edge is of the form $(e \setminus \{u,v\}) \cup \{x,y\}$ for some $\{x,y\} \in E(L)$. Since $\{x,y\}$ is an edge in the bipartite graph $L$ with partition $(\mathcal{P}, \mathcal{Q})$, exactly one of its endpoints must lie in $\mathcal{P}$. As $B \cap V(L) = \mathcal{P}$ by our construction, it follows that $|\{x,y\} \cap B|=1$. The total intersection size with $B$ is therefore $|(e \setminus \{u,v\}) \cap B| + |\{x,y\} \cap B| = |E_u \setminus \{u\}| + 1 = |E_u|$. As this intersection size is constant for all edges in the star, all edges in $S(e \setminus \{u,v\})$ lie in the slice $\mathrm{Slice}(B, |E_u|)$. The minimum degree condition on $L$ ensures $B$ satisfies the size constraints in Definition~\ref{def:slice}. This establishes the existence of such a slice.

Next, we prove uniqueness. Assume $\mathrm{Slice}(B', j)$ is another slice satisfying the lemma's conditions. The condition on the star implies that for any edge $\{x,y\} \in E(L)$, the value $|\{x,y\} \cap B'|$ must be constant. This constant cannot be 0 or 2, as either case would imply that one of the sets $V(L) \cap B'$ or $V(L) \setminus B'$ contains all vertices of $L$, which would result in $|B'| < \frac{1}{3}|V(G)|$ or $|B'| > \frac{2}{3}|V(G)|$, contradicting the size constraints in Definition~\ref{def:slice}. Thus, the constant must be 1. This forces $(V(L) \cap B', V(L) \setminus B')$ to be a bipartition of $L$. Since $L$ is connected, its bipartition is unique, meaning $V(L) \cap B'$ must be either $\mathcal{P}$ or $\mathcal{Q}$. This choice, combined with the requirement that $e \cap B'$ must be either $E_u$ or $E_v$, uniquely determines the set $B'$ up to complementation. Thus, the slice satisfying the conditions is unique.
\end{proof}

The previous lemma guarantees a unique slice for the star associated with a single non-exchangeable pair. The next lemma strengthens this result, showing that this slice is in fact identical for all non-exchangeable pairs within the same edge $e$.

\begin{lemma}
\label{lemma:edge_in_unique_slice}
Let $G$ be a $k$-uniform hypergraph with $\delta_{k-1}(G) > \frac{1}{3} |V(G)|$ that contains no homomorphic copy of $C^k_{\ell}$, where $\ell$ is a fixed integer satisfying $\ell \geq 20k^2$ and $k \nmid \ell$. For any edge $e \in E(G)$, there exists a unique slice such that it is precisely the slice established in Lemma~\ref{lemma:unique_slice_star} for any given pair $\{u,v\} \in E(H_e)$. Consequently, for every pair $\{u,v\} \in E(H_e)$, all edges in the star $S(e \setminus \{u,v\})$ are contained within this slice.
\end{lemma}

\begin{proof}[Proof of Lemma~\ref{lemma:edge_in_unique_slice}]
Uniqueness follows directly from Lemma~\ref{lemma:unique_slice_star}.

To prove existence, since Lemma~\ref{lemma:edge_partition} establishes that $H_e$ is a complete bipartite graph, it suffices to show that any two adjacent edges, for instance $\{u,v\}$ and $\{u,w\}$, generate the same slice via Lemma~\ref{lemma:unique_slice_star}.
Let $(E_1, E_2)$ be the bipartition of the complete bipartite graph $H_e$, with $u \in E_1$ and $v,w \in E_2$. Let $A := e \setminus \{u,v,w\}$ and consider the 3-uniform link hypergraph $F := L(A)$ on the vertex set $V(G) \setminus A$. We define two vertex sets $\mathcal{P}, \mathcal{Q} \subseteq V(F)$ as follows:
\begin{align*}
    \mathcal{P} &:= \{x \in V(F) \colon \exists \{y,z\} \text{ s.t. } \{x,y,z\} \in E(F) \text{ and } \{y,z\} \in E(H_{A \cup \{x,y,z\}})\}, \\
    \mathcal{Q} &:= \{x \in V(F) \colon \exists \{y,z\} \text{ s.t. } \{x,y,z\} \in E(F) \text{ and } \{y,z\} \notin E(H_{A \cup \{x,y,z\}})\}.
\end{align*}
An immediate consequence of these definitions is that $u \in \mathcal{Q}$ while $v,w \in \mathcal{P}$.

\begin{claim}
The pair $(\mathcal{P}, \mathcal{Q})$ forms a partition of $V(F)$.
\end{claim}
\begin{poc}
First, we observe that $\mathcal{P} \cup \mathcal{Q} = V(F)$. This is because the minimum codegree property of $G$ guarantees that for every vertex $x \in V(F)$, there exists at least one edge of the form $\{x,y,z\} \in E(F)$. By definition, this places $x$ in either $\mathcal{P}$ or $\mathcal{Q}$.

Next, we show that $\mathcal{P}$ and $\mathcal{Q}$ are disjoint by contradiction. Assume there is a vertex $x \in \mathcal{P} \cap \mathcal{Q}$. By definition, this implies the existence of hyperedges $e_1 = A \cup \{x,y,z\}$ and $e_2 = A \cup \{x,y',z'\}$, such that $\{y,z\} \in E(H_{e_1})$ while $\{y',z'\} \notin E(H_{e_2})$. The condition $\{y,z\} \in E(H_{e_1})$ means that $\{y,z\}$ is a non-exchangeable pair in the edge $e_1$. By Lemma~\ref{lemma:link_graph_properties}, the corresponding link graph $L(e_1 \setminus \{y,z\}) = L(A \cup \{x\})$ is therefore connected. This implies the existence of a path in $L(A \cup \{x\})$ from $\{y,z\}$ to $\{y',z'\}$. For every pair of adjacent edges in this path, combining each with $A \cup \{x\}$ yields hyperedges that satisfy the conditions of Lemma~\ref{lemma:auxiliary_isomorphism}, thereby establishing an isomorphism between their auxiliary graphs. This chain of isomorphisms from $H_{e_1}$ to $H_{e_2}$ maps $\{y,z\}$ to $\{y',z'\}$. This yields a contradiction, since an edge in $H_{e_1}$ is mapped to a non-edge in $H_{e_2}$.
\end{poc}

A key structural property of the partition is that for any hyperedge $e' \in E(F)$, we have $|e' \cap \mathcal{P}| \in \{0, 2\}$. This can be seen by considering the arrangement of the vertices of $e'$ within the bipartite graph $H_{A\cup e'}$. If all three vertices of $e'$ lie in the same partition of $H_{A\cup e'}$, they form an independent set, placing all three in $\mathcal{Q}$ and none in $\mathcal{P}$. If they are split between the two partitions, there are exactly two edges among them in $H_{A\cup e'}$, which by definition places two vertices in $\mathcal{P}$ and one in $\mathcal{Q}$.

We prove the existence of the required slice by constructing it directly. Let the defining set be $B := \mathcal{Q} \cup E_1$. We now verify that the slice $\mathrm{Slice}(B, |e \cap B|)$ satisfies the conditions of Lemma~\ref{lemma:unique_slice_star} for both pairs $\{u,v\}$ and $\{u,w\}$.

By construction, the sets $e \cap B = E_1$ and $e \setminus B = E_2$ already form the bipartition of $H_e$. It is therefore sufficient to show that both  $S(e \setminus \{u,v\})$ and $S(e \setminus \{u,w\})$ are contained in the slice.

We first verify this for $S(e \setminus \{u,v\})$. An arbitrary hyperedge in this star, $e' := \{x,y\} \cup (e \setminus \{u,v\})$, corresponds to an edge $\{x,y\} \in E(L(e \setminus \{u,v\}))$. This implies the set $\{w,x,y\}$ is a hyperedge in the 3-uniform link hypergraph $F$. Since $w \in \mathcal{P}$, the structural property that any hyperedge of $F$ intersects $\mathcal{P}$ in either 0 or 2 vertices implies that exactly one of $x,y$ is in $\mathcal{P}$ and the other is in $\mathcal{Q}$, which gives $|\{x,y\} \cap B| = 1$. Furthermore, since $u \in E_1 \subseteq B$ and $v \in \mathcal{P} \subseteq V(G) \setminus B$, we also have $|\{u,v\} \cap B| = 1$. The total intersection size of $e'$ with $B$ is therefore given by the following:
\[
|e' \cap B| = |\{x,y\} \cap B| + |(e \setminus \{u,v\}) \cap B| = 1 + (|e \cap B| - |\{u,v\} \cap B|) = 1 + (|e \cap B| - 1) = |e \cap B|.
\]
This confirms that $S(e \setminus \{u,v\})$ is contained in the slice.

An identical argument holds for $S(e \setminus \{u,w\})$, which completes the proof of existence.\end{proof}

\subsection{Step 3: Characterizing Tight Components}
\label{subsec:step3}

The result of Lemma~\ref{lemma:edge_in_unique_slice} establishes a unique slice for each individual edge. This section extends that finding, proving in the following lemma that all edges within a single tight component share a common unique slice and, furthermore, that this slice provides a complete characterization of the component.

\begin{lemma}
\label{lemma:reachable_edges_in_slice}
Let $G$ be a $k$-uniform hypergraph with $\delta_{k-1}(G) > \frac{1}{3} |V(G)|$ that contains no homomorphic copy of $C^k_{\ell}$, where $\ell$ is a fixed integer satisfying $\ell \geq 20k^2$ and $k \nmid \ell$. For any edge $e \in E(G)$, the following hold:
\begin{enumerate}
    \item There exists a single slice $\mathrm{Slice}(B,i)$ that serves as the unique slice guaranteed by Lemma~\ref{lemma:edge_in_unique_slice} for every edge $e' \in \mathcal{C}(e)$. Consequently, $\mathcal{C}(e) \subseteq \mathrm{Slice}(B,i)$.
    \item This slice provides a perfect characterization of the tight component:  $\mathcal{C}(e)=\mathrm{Slice}(B,i)\cap E(G)$. Furthermore, the entire graph $G$ contains no edge $f$ such that $|f \cap B| \in \{i-1, i+1\}$.
\end{enumerate}
\end{lemma}
\begin{proof}[Proof of Lemma~\ref{lemma:reachable_edges_in_slice}]
We prove the first statement by showing that any two adjacent edges in $\mathcal{C}(e)$ share the same unique slice, as the claim for the entire component then follows from the definition of tight connectivity. Let $e_1 = \{x_1, \dots, x_{k-1}, a\}$ and $e_2 = \{x_1, \dots, x_{k-1}, b\}$ be two such adjacent edges.

Our first step is to determine the unique slice associated with $e_1$. We find an edge in the auxiliary graph $H_{e_1}$ incident to $a$, and let it be $\{x_{k-1}, a\}$ without loss of generality. By Lemma~\ref{lemma:unique_slice_star}, the pair $(e_1, \{x_{k-1}, a\})$ then determines a unique slice, which we denote $\mathrm{Slice}(B,i)$.

Next, we show that this same slice, $\mathrm{Slice}(B,i)$, is also the unique slice corresponding to $e_2$. Notice that $S(e_1\setminus \{x_{k-1}, a\})=S(\{x_1, \dots, x_{k-2}\})=S(e_2\setminus \{x_{k-1}, b\})$ is contained in $\mathrm{Slice}(B,i)$.
It suffices to prove that $B$ also induces the bipartition of $H_{e_2}$. We begin by noting that since both $e_1$ and $e_2$ belong to the star $S(\{x_1, \dots, x_{k-2}\})$, both are contained in $\mathrm{Slice}(B,i)$. The isomorphism between $H_{e_1}$ and $H_{e_2}$ (from Lemma~\ref{lemma:auxiliary_isomorphism}) maps $a$ to $b$. Because both edges lie in $\mathrm{Slice}(B,i)$, they have the same intersection size with $B$, which forces $a$ and $b$ to share the same membership status relative to $B$. This ensures the isomorphism preserves the bipartition, confirming that $B$ correctly partitions $H_{e_2}$.

Therefore, $\mathrm{Slice}(B,i)$ fulfills the conditions of Lemma~\ref{lemma:unique_slice_star} for the pair $(e_2, \{x_{k-1}, b\})$ as well, establishing it as the unique slice for $e_2$. Since we have shown that any two adjacent edges share a common unique slice, the property extends to the entire tight component $\mathcal{C}(e)$.

Given that $\mathcal{C}(e)$ is contained within the unique slice $\mathrm{Slice}(B,i)$, we now prove the lemma's second assertion, beginning with the following auxiliary claim. We use $B^c$ to denote $V(G)\setminus B$.

\begin{claim}
\label{claim:reachable-subset-special-intersection}
For any set $A \subseteq V$ with $|A|=k-2$ and $|A \cap B| = i-1$, there exists an edge $e_0 \in \mathcal{C}(e)$ such that $A \subseteq e_0$.
\end{claim}

\begin{poc}
We proceed by contradiction. Assume the claim is false, let $e_0 \in \mathcal{C}(e)\subseteq \mathrm{Slice}(B,i)$ be an edge that maximizes $|e_0 \cap A|$. By assumption, $|e_0 \cap B| = i$, $|e_0 \setminus B| = k-i$ and $A \not\subseteq e_0$.

Our next step is to construct three $(k-1)$-sets whose neighborhoods must intersect. Since $A \not\subseteq e_0$, we can choose a vertex $a \in A \setminus e_0$. We will select two vertices from $e_0 \setminus A$ for our construction. Since $|e_0 \cap B| = i$ while $|e_0 \cap A \cap B| \le |A \cap B| = i-1$, there exists a vertex $u \in (e_0 \setminus A) \cap B$. Since $|e_0 \setminus B| = k-i$ while $|(e_0 \cap A) \setminus B| \le |A \setminus B| = k-i-1$, there exists a vertex $v \in (e_0 \setminus A) \cap B^c$. We then define three $k-1$ sets: $S_1 = e_0 \setminus \{u\}$, $S_2 = e_0 \setminus \{v\}$, and $S_3 = (e_0 \setminus \{u,v\}) \cup \{a\}$. By the codegree condition and the Pigeonhole Principle, at least two of their neighborhoods must intersect.

We now show that two of these neighborhoods, $N(S_1)$ and $N(S_2)$, are in fact disjoint. Any vertex $w \in N(S_1)$ must form an edge $S_1 \cup \{w\}$  which is adjacent to  $e_0$ and thus lies in $\mathcal{C}(e)$, which requires $|(S_1 \cup \{w\}) \cap B| = i$. Since $|S_1 \cap B| = i-1$, it must be that $w \in B$, so $N(S_1) \subseteq B$. A similar argument shows that any vertex in $N(S_2)$ must belong to $B^c$, implying $N(S_2) \subseteq B^c$. Therefore, the two neighborhoods are disjoint.

The disjointness of $N(S_1)$ and $N(S_2)$ forces another pair of neighborhoods to intersect. Without loss of generality, let $w \in N(S_1) \cap N(S_3)$. This allows us to form two new edges, $e_1 = S_1 \cup \{w\}=(e_0 \setminus \{u\}) \cup \{w\}$ and $e_2 = S_3 \cup \{w\}=(e_0 \setminus \{u,v\}) \cup \{a,w\}$. By construction, $|e_0 \cap e_1| = k-1$ and $|e_1 \cap e_2| = k-1$, meaning $e_0$ is adjacent to $e_1$, and $e_1$ is adjacent to $e_2$. This implies that $e_2$ is also in the tight component $\mathcal{C}(e)$.

This new edge $e_2$ leads to the final contradiction. Since $e_2$ is formed from $e_0$ by removing two vertices $u,v$ that are not in $A$ and adding two vertices $a,w$ where $a$ is in $A$, its intersection with $A$ is strictly larger. This contradicts the maximality of $|e_0 \cap A|$ and thus proves the claim.
\end{poc}

Having established Claim~\ref{claim:reachable-subset-special-intersection}, it remains to show that $\mathcal{C}(e)$ contains precisely all edges with intersection size $i$, and that no edges with intersection size $i-1$ or $i+1$ exist. We will prove the following claim, which provides a unified proof for both. The claim directly establishes the characterization of $\mathcal{C}(e)$, while the non-existence of edges $f$ with $|f \cap B| \in \{i-1, i+1\}$ is proved by contradiction: if an edge $f$ with $|f \cap B| \in \{i-1, i+1\}$ existed, the claim would force $f \in \mathcal{C}(e)\subseteq\mathrm{Slice}(B,i)$ and $|f \cap B| = i\notin \{i-1, i+1\}$.

\begin{claim}\label{claim:intersection-implies-reachability}
 Any edge $e' \in E(G)$ satisfying $|e' \cap B| \in \{i-1, i, i+1\}$ must belong to $\mathcal{C}(e)$.
\end{claim}
\begin{poc}
Our goal is to show that $e'$ is in the same tight component as $e$. We select a subset $A \subseteq e'$ of size $k-2$ with $|A \cap B| = i-1$. This selection is possible because $|e' \cap B| \in \{i-1, i, i+1\}$. By Claim~\ref{claim:reachable-subset-special-intersection}, there exists an edge $e_0 \in \mathcal{C}(e)$ containing $A$. Let $e_0 = A \cup \{u,v\}$ and $e' = A \cup \{x,y\}$.

The properties of $e_0$ allow us to show that the link graph $L(A)$ is connected. Since $e_0 \in \mathcal{C}(e)$, the preceding proof of the first statement of this lemma shows that the slice for $e_0$ guaranteed by Lemma~\ref{lemma:edge_in_unique_slice} is also $\mathrm{Slice}(B,i)$. This implies $|e_0 \cap B| = i$. From our choice of $A$ with $|A \cap B| = i-1$, it must be that $|\{u,v\} \cap B| = 1$. By Lemma~\ref{lemma:unique_slice_star}, the auxiliary graph $H_{e_0}$ is bipartite with parts $e_0 \cap B$ and $e_0 \setminus B$. This places $u$ and $v$ in opposite parts of the bipartition, confirming $\{u,v\} \in E(H_{e_0})$. By Lemma~\ref{lemma:link_graph_properties}, this implies that $L(A)$ is connected.

The connectivity of $L(A)$ provides the required connection. The hyperedges $e_0$ and $e'$ correspond to edges $\{u,v\}$ and $\{x,y\}$ in $L(A)$. As $L(A)$ is connected, there is a path of edges $f_0, \dots, f_t$ from $f_0=\{u,v\}$ to $f_t=\{x,y\}$ where any two consecutive edges $f_j, f_{j+1}$ share a vertex. This path induces a sequence of hyperedges $E_j = A \cup f_j$. Because consecutive edges in the path share a vertex, the corresponding hyperedges $E_j$ and $E_{j+1}$ are adjacent. This sequence of adjacent hyperedges shows that $e_0$ and $e'$ are in the same tight component. Since $e_0 \in \mathcal{C}(e)$, it follows that $e' \in \mathcal{C}(e)$. \end{poc}

This completes the proof of this subsection.
\end{proof}

\subsection{Step 4: Completing the Proof of \cref{thm:structural}}
\label{subsec:step4}

In this final section, we complete the proof of \cref{thm:structural}. Building on the characterization of tight components from the previous section, we establish the odd intersection property by showing that the underlying vertex partition is universal across the entire graph. Throughout this section, we use $B^c$ to denote $V(G)\setminus B$. For the reader's convenience, we restate the main theorem before providing the proof.

\medskip
\noindent\textbf{Theorem~\ref{thm:structural}.} \textit{Let $3\le k< \ell$ be integers with $k \nmid \ell$ and $\ell \geqslant 20k^2$.  For every $k$-uniform hypergraph $G$ that contains no homomorphic copy of $C^k_\ell$ and satisfies $\delta_{k-1}(G) > \frac{1}{3} |V (G)|$, there exists a vertex set $B \subseteq V (G)$ with $\frac{|V (G)|}{3} \leqslant |B| \leqslant \frac{2|V (G)|}{3}$ such that for every $e \in E(G)$, both $|e \cap B|$ and $|e\cap B^c|$ are odd. 
In particular, $k$ has to be even.}

\begin{proof}[Proof of \cref{thm:structural}]
The proof begins by analyzing the structure of the tight components of $G$. We associate with each tight component $\mathcal{C}$ a set of $(k-1)$-subsets, which we call its \emph{trace}, defined as follows:
\[T(\mathcal{C})=\bigcup_{f\in \mathcal{C}} \binom{f}{k-1}.\]

We now show that the collection of these distinct traces, $\{T(\mathcal{C})\}$, forms a partition of $\binom{V(G)}{k-1}$.

First, we establish that every $(k-1)$-set is contained in at least one trace. Given any $A \in \binom{V(G)}{k-1}$, the minimum codegree condition $\delta_{k-1}(G) > |V(G)|/3$ ensures that an edge $e = A \cup \{v\}$ exists in $G$. This edge $e$ must belong to some tight component $\mathcal{C}(e)$, which implies that $A$ is in its trace $T(\mathcal{C}(e))$.

Second, we establish that traces of distinct components are disjoint. Assume for contradiction that a set $A$ lies in two distinct traces, $T(\mathcal{C}_1)$ and $T(\mathcal{C}_2)$. This implies the existence of edges $e_1 \in \mathcal{C}_1$ and $e_2 \in \mathcal{C}_2$ that both contain $A$. Since $|e_1 \cap e_2|=|A| \ge k-1$, they are adjacent. By the definition of tight connectivity, they must belong to the same component, contradicting the assumption that $\mathcal{C}_1$ and $\mathcal{C}_2$ were distinct.

Having established both the coverage and disjointness properties, we conclude that $\{T(\mathcal{C})\}$ partitions $\binom{V(G)}{k-1}$.

We first establish an auxiliary claim that provides a uniqueness condition for traces, which will be applied in the proofs of the subsequent claims.

\begin{claim}\label{claim:slice-agreement}
Let $B \subseteq V(G)$ satisfy $k \leq |B| \leq |V(G)|-k$. Suppose that for some subset $D \subseteq V(G)$ and integer $j$, the set $T = \{A \in \binom{V(G)}{k-1} \colon |A \cap D| \in \{j, j-1\}\}$ contains all $(k-1)$-sets intersecting $B$ in exactly $i$ vertices and excludes all $(k-1)$-sets intersecting $B$ in exactly $i+1$ vertices. Then either $(D, j) = (B, i)$ or $(D, j) = (B^c, k-i)$, and consequently $T = \{A \in \binom{V(G)}{k-1} \colon |A \cap B| \in \{i, i-1\}\}$.
\end{claim}

\begin{poc}
The proof proceeds in two main steps. 

First, we show that the set $D$ must be either $B$ or $B^c$. We begin by picking arbitrary vertices $x \in B$ and $y \in B^c$, and fixing a $(k-2)$-set $R$ such that $x, y \notin R$ and $|R \cap B| = i$. We then define two test sets, $A_1 = R \cup \{x\}$ and $A_2 = R \cup \{y\}$. By construction, $|A_1 \cap B| = i+1$ and $|A_2 \cap B| = i$. The premise on $T$ thus implies $A_1 \notin T$ and $A_2 \in T$. Now, assume for contradiction that $x$ and $y$ lie on the same side of the partition defined by $D$. This would mean $|A_1 \cap D| = |A_2 \cap D|$, which in turn forces $A_1$ and $A_2$ to have the same membership status in $T$. This is a contradiction. Therefore, $D$ must separate every vertex of $B$ from every vertex of $B^c$, which implies either $D=B$ or $D=B^c$.

Second, we determine the value of $j$ in each of these two cases.
\begin{itemize}
    \item If $D = B$, the definition of $T$ becomes $\{A \in \binom{V(G)}{k-1} \colon |A \cap B| \in \{j, j-1\}\}$. Since $T$ must contain all sets with $|A \cap B| = i$ but no sets with $|A \cap B| = i+1$, it follows that $j = i$.
    \item If $D = B^c$, the definition of $T$ is $\{A \in \binom{V(G)}{k-1} \colon |A \cap B^c| \in \{j, j-1\}\}$, which is equivalent to $\{A \in \binom{V(G)}{k-1} \colon |A \cap B| \in \{k-j, k-j-1\}\}$. Since $T$ must contain all sets with $|A \cap B| = i$ but no sets with $|A \cap B| = i+1$, it follows that $k-j = i$.
\end{itemize}
In both scenarios, the set $T$ is uniquely determined as $\{A \in \binom{V(G)}{k-1} \colon |A \cap B| \in \{i, i-1\}\}$, which completes the proof.
\end{poc}

The next claim establishes the equivalence between the slice-based characterization of a tight component from Lemma~\ref{lemma:reachable_edges_in_slice} and a precise formulation of its trace. This equivalence will be instrumental for the remainder of the proof.
\begin{claim}\label{claim:trace_form}
Fix a tight component \(\mathcal{C}(e)\), \(B \subseteq V(G)\), and $i \in \mathbb{Z}$. The following are equivalent:
\begin{enumerate}
    \item \(\mathrm{Slice}(B,i)\) is the unique slice that characterizes \(\mathcal{C}(e)\) as defined in Lemma~\ref{lemma:reachable_edges_in_slice}.
    \item \(T(\mathcal{C}(e)) = \{A \in \binom{V(G)}{k-1} \colon |A \cap B| \in \{i, i-1\}\}\).
\end{enumerate}
\end{claim}

\begin{poc}
We prove the equivalence by establishing both implications.

First, we prove that if $\mathrm{Slice}(B,i)$ is the unique slice for $\mathcal{C}(e)$ as defined in Lemma~\ref{lemma:reachable_edges_in_slice}, then its trace must be $T(\mathcal{C}(e)) = \{A \in \binom{V(G)}{k-1} \colon |A \cap B| \in \{i, i-1\}\}$. The proof proceeds by double inclusion. The forward inclusion is straightforward: any $A \in T(\mathcal{C}(e))$ is a subset of some edge $e' \in \mathcal{C}(e)$ with $|e' \cap B|=i$, so $|A \cap B|$ must be $i$ or $i-1$. For the reverse inclusion, consider a set $A$ from $\{A \in \binom{V(G)}{k-1} \colon |A \cap B| \in \{i, i-1\}\}$. The minimum codegree condition ensures that for some vertex $v$, the set $e'' = A \cup \{v\}$ is an edge. Its intersection size $|e'' \cap B|$ can only be $i-1, i,$ or $i+1$. Since Lemma~\ref{lemma:reachable_edges_in_slice} forbids intersection sizes $i-1$ and $i+1$, we must have $|e'' \cap B|=i$, which implies $e'' \in \mathcal{C}(e)$ and thus $A \in T(\mathcal{C}(e))$.

For the converse, we aim to show that if the trace of a component $\mathcal{C}(e)$ is given by $T(\mathcal{C}(e)) = \{A \in \binom{V(G)}{k-1} \colon |A \cap B| \in \{i, i-1\}\}$, then $\mathrm{Slice}(B,i)$ must be the unique characterizing slice for $\mathcal{C}(e)$.
By Lemma~\ref{lemma:reachable_edges_in_slice}, a unique characterizing slice for $\mathcal{C}(e)$ exists, which we denote by $\mathrm{Slice}(D,j)$. This implies the trace is also given by $T(\mathcal{C}(e)) = \{A \in \binom{V(G)}{k-1} \colon |A \cap D| \in \{j, j-1\}\}$.
Equating the two expressions for the trace yields:
\[ \{A \in \binom{V(G)}{k-1} \colon |A \cap B| \in \{i, i-1\}\} = \{A \in \binom{V(G)}{k-1} \colon |A \cap D| \in \{j, j-1\}\}. \]
By Claim~\ref{claim:slice-agreement}, this equality implies that the parameters are related by either $(D, j) = (B, i)$ or $(D, j) = (B^c, k-i)$. Both pairs define the same slice, $\mathrm{Slice}(B,i)$. Therefore, $\mathrm{Slice}(B,i)$ is the unique slice that characterizes $\mathcal{C}(e)$, completing the proof.
\end{poc}

The following claim establishes a recursive relationship between traces. It shows that the existence of a trace defined by intersection sizes $\{i, i-1\}$ with a set $B$ guarantees the existence of another trace defined by sizes $\{i-2, i-3\}$ with the same set $B$.
\begin{claim} \label{claim:recursion}
For a fixed vertex set $B$ and any integer $i \geq 2$, if $\{ A \in \binom{V(G)}{k-1} \colon |A \cap B| \in \{i, i-1\}\}$ is a trace, then $\{ A \in \binom{V(G)}{k-1} \colon |A \cap B| \in \{i-2, i-3\}\}$ is also a trace.
\end{claim}

\begin{poc}
The proof proceeds in three steps.

First, we show that all $(k-1)$-sets intersecting $B$ in exactly $i-2$ vertices belong to a single trace. Let this collection of sets be
$$ S = \left\{ A \in \binom{V(G)}{k-1} \colon |A \cap B| = i-2 \right\}. $$
To prove this, it suffices to show that any two sets $A_1, A_2 \in S$ sharing $k-2$ vertices are in the same trace. The existence of a trace of the form $\{A' \in \binom{V(G)}{k-1} \colon |A' \cap B| \in \{i, i-1\}\}$ implies, by Claim~\ref{claim:trace_form} and Lemma~\ref{lemma:reachable_edges_in_slice}, that no edge in the entire graph $G$ can intersect $B$ in exactly $i-1$ vertices. Therefore, the neighborhoods $N(A_1)$ and $N(A_2)$ must both be subsets of $B^c$. The minimum codegree condition implies $|N(A_1)|, |N(A_2)| > |V(G)|/3$. Since $|B^c| \le 2|V(G)|/3$, the sum of the neighborhood sizes exceeds $|B^c|$, guaranteeing their intersection is non-empty. Choosing a common neighbor $w \in N(A_1) \cap N(A_2)$ yields two adjacent edges, $A_1 \cup \{w\}$ and $A_2 \cup \{w\}$, which places $A_1$ and $A_2$ in the same trace.

Next, we identify this trace, which we call $T$. From the first step, we know $T$ contains all of $S$. By premise, a different trace already exists containing all sets that intersect $B$ in $i-1$ vertices. As distinct traces are disjoint, $T$ cannot contain any set that intersects $B$ in $i-1$ vertices.

Finally, we determine the precise form of $T$. We have shown that $T$ contains every set in $S$ (those with intersection size $i-2$) but no set with intersection size $i-1$. By Claim~\ref{claim:slice-agreement}, these conditions uniquely determine the trace's form. Therefore, $T$ must be the set $\{A \in \binom{V(G)}{k-1} \colon |A \cap B| \in \{i-2, i-3\}\}$, completing the proof.
\end{poc}

\begin{claim} 
For a fixed vertex set $B$ and integer $i$, if $\{ A \in \binom{V(G)}{k-1} \colon |A \cap B| \in \{i, i-1\}\}$ is a trace, then $k$ is even, $i$ is odd and all traces are $$\{ A \in \binom{V(G)}{k-1} \colon |A \cap B| \in \{2m-1, 2m-2\}\}\quad \text{for } m = 1, 2, \ldots,  \frac{k}{2}.$$
\end{claim}
\begin{poc}
    By Definition~\ref{def:slice} and Claim~\ref{claim:trace_form}, we know that $1\leq i\leq k-1$.
    If $i\geq2 $, by Claim~\ref{claim:recursion}, we have  $\{ A \in \binom{V(G)}{k-1} \colon |A \cap B| \in \{i-2, i-3\}\}$ is also a trace. Recursive application of Claim~\ref{claim:recursion} yields that $\{ A \in \binom{V(G)}{k-1} \colon |A \cap B| \in \{j, j-1\}\}$ is a trace for some $j\in\{0,1\}$. But by Definition~\ref{def:slice} and Claim~\ref{claim:trace_form}, $\{ A \in \binom{V(G)}{k-1} \colon |A \cap B| \in \{0, -1\}\}$ cannot be a trace. Therefore $j$ must be $1$, so $i$ is odd  and $\{ A \in \binom{V(G)}{k-1} \colon |A \cap B| \in \{2m-1, 2m-2\}\}$ is a trace for $m = 1, 2, \ldots,  \frac{i+1}{2}$.

    Notice that $\{ A \in \binom{V(G)}{k-1} \colon |A \cap B| \in \{i, i-1\}\}=\{ A \in \binom{V(G)}{k-1} \colon |A \cap B^c| \in \{k-i,k-i-1\}\}$ is a trace. Thus, by an analogous argument on $k-i+1$ and $B^c$, we observe that  $k-i$ is odd and $\{ A \in \binom{V(G)}{k-1} \colon |A \cap B^c| \in \{2m-1, 2m-2\}\}$ is a trace for $m = 1, 2, \ldots,  \frac{k-i+1}{2}$.

    Consequently, $k=(k-i)+i$ is even and all the traces are $\{ A \in \binom{V(G)}{k-1} \colon |A \cap B| \in \{2m-1, 2m-2\}\}\quad \text{for } m = 1, 2, \ldots,  \frac{k}{2}.$
\end{poc}

Finally, we relate this characterization of traces to the structure of the edges in $G$. Since these traces cover all of $\binom{V(G)}{k-1}$, the trace of any tight component $T(\mathcal{C}(e))$ must be one of these traces. Specifically, for any edge $e \in E(G)$, there must exist an integer $m$ such that $T(\mathcal{C}(e)) = \{A \in \binom{V(G)}{k-1} \colon |A \cap B| \in \{2m-1, 2m-2\}\}$. By Claim~\ref{claim:trace_form}, this implies that $\mathrm{Slice}(B, 2m-1)$ is the unique slice that characterizes the component $\mathcal{C}(e)$. By the definition of a characterizing slice from Lemma~\ref{lemma:reachable_edges_in_slice}, this means $e \in \mathrm{Slice}(B, 2m-1)$. Consequently, $|e \cap B| = 2m-1$ and $|e \cap B^c| = k-2m+1$ are both odd numbers. The size constraint $\frac{|V(G)|}{3} \leq |B| \leq \frac{2|V(G)|}{3}$ is a direct consequence of the properties of slices and traces established in Definition~\ref{def:slice} and Claim~\ref{claim:trace_form}.
\end{proof}

\section{Concluding Remarks}\label{sec:concluding}

In this paper, we establish the upper bound $\gamma(C^k_\ell) \le 1/3$ when $2 \nmid k/\gcd(k,\ell)$. This bound is a consequence of Theorem~\ref{thm:structural}, which characterizes the structure of $n$-vertex hypergraphs containing no homomorphic copy of $C^k_\ell$ with codegree exceeding $n/3$. The proof of Theorem~\ref{thm:structural} relies critically on the assumption $\delta_{k-1}(G) > n/3$ and does not extend to a codegree of $(1/3 - \varepsilon)n$ for any $\varepsilon > 0$. This suggests that the $1/3$ upper bound might be sharp, which motivates the following question.

\begin{question}\label{q:main}
Does the following hold for any integer $k \ge 3$ and sufficiently large $\ell$ with $k \nmid \ell$,
\[
\gamma(C^k_\ell) =
\begin{cases}
    1/2 & \text{if } 2 \mid \frac{k}{\gcd(k,\ell)} \\
    1/3 & \text{if } 2 \nmid \frac{k}{\gcd(k,\ell)}
\end{cases}
~?\]
\end{question}

The main challenge in answering Question~\ref{q:main} is to establish a matching $1/3$ lower bound, which in our approach requires the construction of $(k,\ell;3)$-families.
Our method currently fails in cases where $\gcd(k,\ell)\geq 4$ and $k/\gcd(k,\ell)$ is not divisible by $2$ or $3$.
The smallest such cases occur when $\gcd(k,\ell)=4$ and $k/\gcd(k,\ell)\in\{5,7\}$.
To investigate these instances, we conducted a computational search for the required $(k,\ell;3)$-families.
When $\gcd(k,\ell)=3$, we found $432$ and $540{,}422$ feasible families for $k/\gcd(k,\ell)=5$ and $7$, respectively.
In contrast, when $\gcd(k,\ell)=4$, a complete depth-first search confirmed that no feasible families exist for either $k/\gcd(k,\ell)=5$ or $7$.
This strongly suggests that the obstruction is fundamental.
Resolving these remaining cases likely requires new ideas, such as probabilistic constructions given in~\cite{CN2001} and~\cite{FPVV2023}.

Despite the limitations of our construction method, we believe it may still offer new insights into the codegree Turán density $\gamma(F)$ for other hypergraphs $F$.
In particular, the $(k,\ell;d)$-family framework is especially well suited to settings in which the extremal $F$-free hypergraphs arise as blow-ups of a fixed combinatorial structure.

Another perspective arising from $(k,\ell;d)$-families is that, for some pairs $(k,\ell)$, the extremal hypergraph attaining $\gamma(C^k_\ell)$ appears to be unique, while for many others there exist multiple non-isomorphic extremal examples (as indicated by the discussion above).
It would be interesting to characterize all extremal hypergraphs for $\gamma(C^k_\ell)$ and to study their (multi-)stability properties.

For $C^{k-}_\ell$, in the exceptional cases of \cref{thm:minus_edge_main_2} (that is, when $\ell \equiv \pm 2 \pmod k$ or $3\ell \equiv \pm 1, \pm 2 \pmod k$), a highly technical case-by-case analysis might yield an improved lower bound of $1/4$ for sufficiently large $k$.
However, we do not pursue this direction, as such a bound is unlikely to be tight.
As in the case of $C^k_\ell$, determining the precise codegree Turán density for the remaining cases of $C^{k-}_\ell$ remains an interesting open problem.

\section*{Acknowledgements}
We would like to thank Mingze Li for helpful discussions at an early stage of this research. We are also grateful to Jun Gao for his careful reading of the manuscript and for providing several valuable comments that helped improve the presentation.

\bibliographystyle{abbrv} 
\bibliography{codegree} 

@article {E1964,
    AUTHOR = {Erd{\H{o}}s, P.},
     TITLE = {On extremal problems of graphs and generalized graphs},
   JOURNAL = {Israel J. Math.},
  FJOURNAL = {Israel Journal of Mathematics},
    VOLUME = {2},
      YEAR = {1964},
     PAGES = {183--190},
      ISSN = {0021-2172},
   MRCLASS = {05.40},
  MRNUMBER = {183654},
MRREVIEWER = {A.\ H.\ Stone},
       DOI = {10.1007/BF02759942},
       URL = {https://doi.org/10.1007/BF02759942},
}

@article {HLS2021,
    AUTHOR = {Han, Jie and Lo, Allan and Sanhueza-Matamala, Nicol\'as},
     TITLE = {Covering and tiling hypergraphs with tight cycles},
   JOURNAL = {Combin. Probab. Comput.},
  FJOURNAL = {Combinatorics, Probability and Computing},
    VOLUME = {30},
      YEAR = {2021},
    NUMBER = {2},
     PAGES = {288--329},
      ISSN = {0963-5483,1469-2163},
   MRCLASS = {05C65 (05C70 05D99)},
  MRNUMBER = {4225789},
MRREVIEWER = {Anna\ A.\ Taranenko},
       DOI = {10.1017/S0963548320000449},
       URL = {https://doi.org/10.1017/S0963548320000449},
}

@misc{M2024,
 author = {Ma, Jie},
 title = {On codegree {{Tur{\'a}n}} density of the 3-uniform tight cycle ${{C_{11}}}$},
 howpublished = {Communications in Mathematics and Statistics},
 year = {to appear}
}

@article {PSS2023,
    AUTHOR = {Piga, Sim\'on and Sales, Marcelo and Sch\"ulke, Bjarne},
     TITLE = {The codegree {{Tur{\'a}n}} density of tight cycles minus one edge},
   JOURNAL = {Combin. Probab. Comput.},
  FJOURNAL = {Combinatorics, Probability and Computing},
    VOLUME = {32},
      YEAR = {2023},
    NUMBER = {6},
     PAGES = {881--884},
      ISSN = {0963-5483,1469-2163},
   MRCLASS = {05C35 (05C65 05D99)},
  MRNUMBER = {4653729},
MRREVIEWER = {J\'ozsef\ Balogh},
       DOI = {10.1017/s0963548323000196},
       URL = {https://doi.org/10.1017/s0963548323000196},
}

@article {PSS2024,
    AUTHOR = {Piga, Sim\'on and Sanhueza-Matamala, Nicol\'as and Schacht, Mathias},
     TITLE = {The codegree {T}ur\'an density of 3-uniform tight cycles},
   JOURNAL = {J. Combin. Theory Ser. B},
  FJOURNAL = {Journal of Combinatorial Theory. Series B},
    VOLUME = {176},
      YEAR = {2026},
     PAGES = {1--6},
      ISSN = {0095-8956,1096-0902},
   MRCLASS = {05C35 (05C65)},
  MRNUMBER = {4943139},
       DOI = {10.1016/j.jctb.2025.07.007},
       URL = {https://doi.org/10.1016/j.jctb.2025.07.007},
}

@misc{S2025,
 author = {Sarkies, James},
 title = {Hypergraphs of arbitrary uniformity with vanishing codegree {{Tur{\'a}n}} density},
 year = {2025},
 howpublished = {Preprint, {arXiv}:2503.23591 [math.{CO}] (2025)},
 url = {https://arxiv.org/abs/2503.23591},
 arXiv = {arXiv:2503.23591}
}

@article {MZ2007,
    AUTHOR = {Mubayi, Dhruv and Zhao, Yi},
     TITLE = {Co-degree density of hypergraphs},
   JOURNAL = {J. Combin. Theory Ser. A},
  FJOURNAL = {Journal of Combinatorial Theory. Series A},
    VOLUME = {114},
      YEAR = {2007},
    NUMBER = {6},
     PAGES = {1118--1132},
      ISSN = {0097-3165,1096-0899},
   MRCLASS = {05C65 (05C35 05D05)},
  MRNUMBER = {2337241},
MRREVIEWER = {J\'ozsef\ Balogh},
       DOI = {10.1016/j.jcta.2006.11.006},
       URL = {https://doi.org/10.1016/j.jcta.2006.11.006},
}

@incollection {K2011,
    AUTHOR = {Keevash, Peter},
     TITLE = {Hypergraph {T}ur\'an problems},
 BOOKTITLE = {Surveys in combinatorics 2011},
    SERIES = {London Math. Soc. Lecture Note Ser.},
    VOLUME = {392},
     PAGES = {83--139},
 PUBLISHER = {Cambridge Univ. Press, Cambridge},
      YEAR = {2011},
      ISBN = {978-1-107-60109-3},
   MRCLASS = {05-02 (05C65)},
  MRNUMBER = {2866732},
}

@article {ES1966,
    AUTHOR = {Erd\H{o}s, P. and Simonovits, M.},
     TITLE = {A limit theorem in graph theory},
   JOURNAL = {Studia Sci. Math. Hungar.},
  FJOURNAL = {Studia Scientiarum Mathematicarum Hungarica. Combinatorics,
              Geometry and Topology (CoGeTo)},
    VOLUME = {1},
      YEAR = {1966},
     PAGES = {51--57},
      ISSN = {0081-6906,1588-2896},
   MRCLASS = {05.40},
  MRNUMBER = {205876},
MRREVIEWER = {W.\ Moser},
}

@article {ES1946,
    AUTHOR = {Erd\H{o}s, P. and Stone, A. H.},
     TITLE = {On the structure of linear graphs},
   JOURNAL = {Bull. Amer. Math. Soc.},
  FJOURNAL = {Bulletin of the American Mathematical Society},
    VOLUME = {52},
      YEAR = {1946},
     PAGES = {1087--1091},
      ISSN = {0002-9904},
   MRCLASS = {56.0X},
  MRNUMBER = {18807},
MRREVIEWER = {H.\ S. M. Coxeter},
       DOI = {10.1090/S0002-9904-1946-08715-7},
       URL = {https://doi.org/10.1090/S0002-9904-1946-08715-7},
}

@article {M2005,
    AUTHOR = {Mubayi, Dhruv},
     TITLE = {The co-degree density of the {F}ano plane},
   JOURNAL = {J. Combin. Theory Ser. B},
  FJOURNAL = {Journal of Combinatorial Theory. Series B},
    VOLUME = {95},
      YEAR = {2005},
    NUMBER = {2},
     PAGES = {333--337},
      ISSN = {0095-8956,1096-0902},
   MRCLASS = {05C65 (05C35)},
  MRNUMBER = {2171370},
MRREVIEWER = {J\'ozsef\ Balogh},
       DOI = {10.1016/j.jctb.2005.06.001},
       URL = {https://doi.org/10.1016/j.jctb.2005.06.001},
}

@article {FPVV2023,
    AUTHOR = {Falgas-Ravry, Victor and Pikhurko, Oleg and Vaughan, Emil and
              Volec, Jan},
     TITLE = {The codegree threshold of {$K^-_4$}},
   JOURNAL = {J. Lond. Math. Soc. (2)},
  FJOURNAL = {Journal of the London Mathematical Society. Second Series},
    VOLUME = {107},
      YEAR = {2023},
    NUMBER = {5},
     PAGES = {1660--1691},
      ISSN = {0024-6107,1469-7750},
   MRCLASS = {05C35},
  MRNUMBER = {4585299},
       DOI = {10.1112/jlms.12722},
       URL = {https://doi.org/10.1112/jlms.12722},
}

@article {FMPV2015,
    AUTHOR = {Falgas-Ravry, Victor and Marchant, Edward and Pikhurko, Oleg
              and Vaughan, Emil R.},
     TITLE = {The codegree threshold for 3-graphs with independent
              neighborhoods},
   JOURNAL = {SIAM J. Discrete Math.},
  FJOURNAL = {SIAM Journal on Discrete Mathematics},
    VOLUME = {29},
      YEAR = {2015},
    NUMBER = {3},
     PAGES = {1504--1539},
      ISSN = {0895-4801,1095-7146},
   MRCLASS = {05D05 (05C35 05C65)},
  MRNUMBER = {3384831},
MRREVIEWER = {Anant\ P.\ Godbole},
       DOI = {10.1137/130926997},
       URL = {https://doi.org/10.1137/130926997},
}

@article {KZ2007,
    AUTHOR = {Keevash, Peter and Zhao, Yi},
     TITLE = {Codegree problems for projective geometries},
   JOURNAL = {J. Combin. Theory Ser. B},
  FJOURNAL = {Journal of Combinatorial Theory. Series B},
    VOLUME = {97},
      YEAR = {2007},
    NUMBER = {6},
     PAGES = {919--928},
      ISSN = {0095-8956,1096-0902},
   MRCLASS = {05C35 (05C65)},
  MRNUMBER = {2354709},
MRREVIEWER = {Ivan\ Pashov},
       DOI = {10.1016/j.jctb.2007.01.004},
       URL = {https://doi.org/10.1016/j.jctb.2007.01.004},
}

@article {DLLWY2025,
    AUTHOR = {Ding, Laihao and Lamaison, Ander and Liu, Hong and Wang,
              Shuaichao and Yang, Haotian},
     TITLE = {On 3-graphs with vanishing codegree {T}ur\'an density},
   JOURNAL = {J. Lond. Math. Soc. (2)},
  FJOURNAL = {Journal of the London Mathematical Society. Second Series},
    VOLUME = {112},
      YEAR = {2025},
    NUMBER = {3},
     PAGES = {Paper No. e70281, 25},
      ISSN = {0024-6107,1469-7750},
   MRCLASS = {05C65 (05C35)},
  MRNUMBER = {4953716},
       DOI = {10.1112/jlms.70281},
       URL = {https://doi.org/10.1112/jlms.70281},
}

@article {ZG2021,
    AUTHOR = {Zhang, Tao and Ge, Gennian},
     TITLE = {On the codegree density of {$PG_m(q)$}},
   JOURNAL = {SIAM J. Discrete Math.},
  FJOURNAL = {SIAM Journal on Discrete Mathematics},
    VOLUME = {35},
      YEAR = {2021},
    NUMBER = {3},
     PAGES = {1548--1556},
      ISSN = {0895-4801,1095-7146},
   MRCLASS = {05C35 (05C65)},
  MRNUMBER = {4279925},
MRREVIEWER = {Ruifang\ Liu},
       DOI = {10.1137/20M1385512},
       URL = {https://doi.org/10.1137/20M1385512},
}

@article {CN2001,
    AUTHOR = {Czygrinow, Andrzej and Nagle, Brendan},
     TITLE = {A note on codegree problems for hypergraphs},
   JOURNAL = {Bull. Inst. Combin. Appl.},
  FJOURNAL = {Bulletin of the Institute of Combinatorics and its
              Applications},
    VOLUME = {32},
      YEAR = {2001},
     PAGES = {63--69},
      ISSN = {1183-1278,2689-0674},
   MRCLASS = {05C65},
  MRNUMBER = {1829685},
}

@misc{S2024,
 author = {Sankar, Maya},
 title = {The {Tur{\'a}n} {Density} of 4-{Uniform} {Tight} {Cycles}},
 year = {2024},
 howpublished = {Preprint, {arXiv}:2411.01782 [math.{CO}] (2024)},
 url = {https://arxiv.org/abs/2411.01782},
 arXiv = {arXiv:2411.01782}
}

@article {KLP2024,
    AUTHOR = {Kam{\v{ c}}ev, Nina and Letzter, Shoham and Pokrovskiy, Alexey},
     TITLE = {The {T}ur\'an density of tight cycles in three-uniform
              hypergraphs},
   JOURNAL = {Int. Math. Res. Not. IMRN},
  FJOURNAL = {International Mathematics Research Notices. IMRN},
      YEAR = {2024},
    VOLUME = {2024},
    NUMBER = {6},
     PAGES = {4804--4841},
      ISSN = {1073-7928,1687-0247},
   MRCLASS = {05C35 (05C65)},
  MRNUMBER = {4721657},
       DOI = {10.1093/imrn/rnad177},
       URL = {https://doi.org/10.1093/imrn/rnad177},
}

@misc{BLLP2025,
 author = {Bodn{\'a}r, Levente and Le{\'o}n, Jared and Liu, Xizhi and Pikhurko, Oleg},
 title = {The {Tur{\'a}n} density of short tight cycles},
 year = {2025},
 howpublished = {Preprint, {arXiv}:2506.03223 [math.{CO}] (2025)},
 url = {https://arxiv.org/abs/2506.03223},
 arXiv = {arXiv:2506.03223}
}

@article {BL2024,
    AUTHOR = {Balogh, J{\'o}zsef and Luo, Haoran},
     TITLE = {Tur\'an density of long tight cycle minus one hyperedge},
   JOURNAL = {Combinatorica},
  FJOURNAL = {Combinatorica. An International Journal on Combinatorics and
              the Theory of Computing},
    VOLUME = {44},
      YEAR = {2024},
    NUMBER = {5},
     PAGES = {949--976},
      ISSN = {0209-9683,1439-6912},
   MRCLASS = {05C65 (05C35 05C38 05D05)},
  MRNUMBER = {4805884},
       DOI = {10.1007/s00493-024-00099-y},
       URL = {https://doi.org/10.1007/s00493-024-00099-y},
}

@misc{BLLP2025minus,
 author = {Bodn{\'a}r, Levente and Le{\'o}n, Jared and Liu, Xizhi and Pikhurko, Oleg},
 title = {The {Tur{\'a}n} density of the tight 5-cycle minus one edge},
 year = {2025},
 howpublished = {Preprint, {arXiv}:2412.21011 [math.{CO}] (2025)},
 url = {https://arxiv.org/abs/2412.21011},
 arXiv = {arXiv:2412.21011}
}

@misc{LMP2024,
 author = {Lidick{\'y}, Bernard and Mattes, Connor and Pfender, Florian},
 title = {The hypergraph {Tur{\'a}n} densities of tight cycles minus an edge},
 year = {2024},
 howpublished = {Preprint, {arXiv}:2409.14257 [math.{CO}] (2024)},
 url = {https://arxiv.org/abs/2409.14257},
 arXiv = {arXiv:2409.14257}
}

@incollection {BCL2022,
    AUTHOR = {Balogh, J\'ozsef and Clemen, Felix Christian and Lidick\'y,
              Bernard},
     TITLE = {Hypergraph {T}ur\'an problems in {$\ell_2$}-norm},
 BOOKTITLE = {Surveys in combinatorics 2022},
    SERIES = {London Math. Soc. Lecture Note Ser.},
    VOLUME = {481},
     PAGES = {21--63},
 PUBLISHER = {Cambridge Univ. Press, Cambridge},
      YEAR = {2022},
      ISBN = {978-1-009-09622-5},
   MRCLASS = {05C35 (05C65)},
  MRNUMBER = {4421399},
}

\appendix
\section{Proof of \cref{thm:minus_edge_main_2}}\label{sec: appendix1}
In this section, we prove \cref{thm:minus_edge_main_2} by building upon and modifying the $(k,\ell;d)$-family framework, \cref{thm:existence}, and \cref{thm:lower_bound_part1}.

Recall that in the proof of \cref{thm:lower_bound_part1}, we prove $\sum_{j=1}^{k} y_j=v_C$ for the case $C^{k-}_\ell$ when $\gcd(k,\ell)>1$. When $\gcd(k,\ell)=1$ we are not able to prove this, instead we can only show that $\sum_{j=1}^{k} y_j=y_1-y_{k+1}+\sum_{j=2}^{k+1} y_j=y_1-y_{k+1}+v_C\in\{v_C-1,v_C,v_C+1\}$. This motivates the following definition.

\begin{definition}\label{def:stable_kld_family}
A $(k,\ell;d)$-family $\mathcal{T}$ is called \emph{stable} if for every connected component $C$ of its type-graph $G_{\mathcal{T}}$, the corresponding value $v_C$ satisfies $\ell v_C  \not\equiv \pm 1 \pmod k$.
\end{definition}

This new definition allows us to establish the following modified version of \cref{thm:lower_bound_part1}.

\begin{theorem}\label{thm:lower_bound_part2}
If $\gcd(k,\ell)=1$ and there exists a stable $(k,\ell;d)$-family $\mathcal{T}$, then $$\gamma(C^{k-}_\ell)\geq \frac{1}{d}.$$
\end{theorem}
\begin{proof}
    The construction of $H$, the assumption that there exists a copy of $C^{k-}_\ell$ in $H$ and the definition of $v_i$, $C$ and $y_i$ are identical to the proof of \cref{thm:lower_bound_part1}. Now instead of $\sum_{j=1}^{k} y_j=v_C$, we know that $\sum_{j=1}^{k} y_j=y_1-y_{k+1}+\sum_{j=2}^{k+1} y_j=y_1-y_{k+1}+v_C$. Recall that for any $t \in [\ell]\setminus\{1\}$, $\sum_{j=t}^{t+k-1} y_j=v_C$. Summing all $\ell$ equations gives
    $$k\sum_{j=1}^{\ell} y_j=\sum_{t\in[\ell]}\sum_{j=t}^{t+k-1} y_j=\ell v_C+y_1-y_{k+1}\in\{\ell v_C-1,\ell v_C,\ell v_C+1\} .$$
    However, as $\mathcal{T}$ is stable, $\ell v_C  \not\equiv \pm 1 \pmod k$, thus $k\sum_{j=1}^{\ell} y_j=\ell v_C$. This means $\frac{\ell}{\gcd(k,\ell)}\cdot v_C =\frac{k}{\gcd(k,\ell)}\cdot \sum_{j=1}^{\ell} y_j$. As $\frac{k}{\gcd(k,\ell)}$ and $\frac{\ell}{\gcd(k,\ell)}$ are coprime, $v_C$ must be divisible by $\frac{k}{\gcd(k,\ell)}$, a contradiction.
\end{proof}
To prove \cref{thm:minus_edge_main_2}, it suffices to prove the following modified version of \cref{thm:existence}.

\begin{theorem}\label{thm:existence_3}
Let $k,\ell$ be integers such that $\ell \not\equiv 0, \pm 1 \pmod k$ and $\gcd(k,\ell)=1$.
If $\ell \not\equiv \pm 2 \pmod k \text{ and } 3\ell \not\equiv \pm 1, \pm 2 \pmod k$, then a stable $(k,\ell;3)$-family exists.
\end{theorem}

\begin{proof}[Proof of  \cref{thm:existence_3}]
The condition $\ell \not\equiv 0, 1, -1 \pmod k$ implies $k \ge 4$. The proof follows the local replacement strategy. A key difference from the proof of \cref{thm:existence} is that the problematic types are not necessarily well-separated, which means their replacement sets can interact. The additional hypotheses on $\ell$ in \cref{thm:existence_3} are introduced precisely to resolve these potential conflicts on a case-by-case basis.

\cref{def:stable_kld_family} requires that for any connected component $C$ of the type-graph, the corresponding invariant $v_C$ must satisfy $\ell v_C \not\equiv \pm 1 \pmod k$. Since $\gcd(k,\ell)=1$, there exists a unique integer $\alpha \in \{1, \dots, k-1\}$ satisfying $\ell \cdot \alpha \equiv 1 \pmod k$, which also implies $\gcd(k, \alpha)=1$. Thus this is equivalent to requiring that $v_C$ is not congruent to $\alpha$ or $k-\alpha \pmod{k}$.

Our strategy is to first define a small set of potentially problematic types, $\mathcal{F} \subseteq \mathcal{T}_3^k$. These are types whose coordinates are structured in a way that makes them susceptible to violating the stability condition. For a 3-tuple of integers $(v_1, v_2, v_3)$, let $\text{Perm}(v_1, v_2, v_3)$ denote the set of all distinct types obtained by permuting its coordinates. We define $\mathcal{F}$ as:
\[ \mathcal{F} := \text{Perm}(k,0,0) \cup \text{Perm}(\alpha, k-\alpha, 0). \]
Note that $\text{Perm}(k,0,0)$ contains 3 types, while $\text{Perm}(\alpha, k-\alpha, 0)$ contains up to 6 types.

We begin with the base family 
\[ \mathcal{B}_3^k := \left\{ \vec{x} = (x_1, x_2, x_3) \in \mathcal{T}_3^k \colon x_1 + 2x_2 \equiv 1 \pmod 3 \right\}, \]
By Claim \ref{claim:base_family_props}, this family satisfies the first property in Definition \ref{def:kld_family} and its connected components are isolated vertices. The set of problematic types is $\mathcal{P} := \mathcal{B}_3^k \cap \mathcal{F}$.
For each problematic type $\vec{x} \in \mathcal{P}$, we define a replacement set $\mathcal{R}_{\vec{x}}$. The final family will be $\mathcal{T} := (\mathcal{B}_3^k \setminus \mathcal{P}) \cup (\bigcup_{\vec{x} \in \mathcal{P}} \mathcal{R}_{\vec{x}})$. The replacement rules are defined by symmetry based on the following examples:
\begin{itemize}
    \item If $\vec{x} = (k,0,0)$, its replacement is $\mathcal{R}_{(k,0,0)} := \{ (k-1,1,0) \}$. The replacements for $(0,k,0)$ and $(0,0,k)$ are defined by cyclically permuting the coordinates.
    \item If $\vec{x} = (\alpha, k-\alpha, 0)$, its replacement is $\mathcal{R}_{(\alpha, k-\alpha, 0)} := \{ (\alpha-1, k-\alpha+1, 0), (\alpha+1, k-\alpha-1, 0) \}$. The replacements for other types in $\text{Perm}(\alpha, k-\alpha, 0)$ are defined analogously.
\end{itemize}

We construct our family as $$\mathcal{T} := (\mathcal{B}_3^k \setminus \mathcal{P}) \cup (\bigcup_{\vec{x} \in \mathcal{P}} \mathcal{R}_{\vec{x}}).$$ 

We illustrate the construction for the case $k=9$ and $\alpha=2$ in Figure~\ref{fig:final_two_figures_large_scale}. The figure on the left displays the base family $\mathcal{B}_3^k$, where types are marked in red, and the set of problematic types is $\mathcal{P}$, which are circled in blue. The figure on the right shows the final family $\mathcal{T}$. In this figure, the problematic types $\mathcal{P}$ have been removed and are circled with dashed red lines, while their corresponding replacement types are added and circled with solid red lines.
\begin{figure}[htbp]
  \centering
  \begin{minipage}{0.48\textwidth} % Adjusted width for two figures
      \centering
      \begin{tikzpicture}[scale=0.82] % Increased scale
        \def\k{9}
        \def\alpha_val{2}
        
        \foreach \y in {0,...,\numexpr\k-1} {
          \foreach \x in {\y,...,\numexpr\k-1} {\fill[blue!25] (\x - 0.5*\y, \y*0.866) -- (\x+1 - 0.5*\y, \y*0.866) -- (\x+0.5 - 0.5*\y, \y*0.866+0.866) -- cycle;}
        }
        
        \foreach \y in {0,...,\k} {
          \draw[thick] (0.5*\y, {0.866*\y}) -- ({\k - 0.5*\y}, {0.866*\y});
          \ifnum \y<\k \foreach \x in {\y,...,\numexpr\k-1} {\draw[thick] ({\x+1 - 0.5*\y}, {0.866*\y}) -- ({\x+1- 0.5*(\y+1)}, {0.866*(\y+1)});} \fi
          \ifnum \y>0 \foreach \x in {\y,...,\k} {\draw[thick] ({\x - 0.5*\y}, {0.866*\y}) -- ({\x - 1 - 0.5*(\y-1)}, {0.866*(\y-1)});} \fi
        }
        
        \foreach \y in {0,...,\k} {
          \foreach \x in {\y,...,\k} {
            \pgfmathtruncatemacro{\check}{mod(-\x + \k + 2*\y - 1, 3)}
            \ifnum \check = 0 \draw[fill=red, thick] (\x - \y*0.5, \y*0.866) circle (5pt);
            \else \draw[fill=white, thick, draw=black] (\x - \y*0.5, \y*0.866) circle (5pt); \fi
          }
        }
        \node[above, yshift=2mm] at ({0.5*\k}, {\k*0.866}) {$(0,\k,0)$};
        \node[below, xshift=1mm, yshift=-2mm] at (0,0) {$(\k,0,0)$};
        \node[below, xshift=-1mm, yshift=-2mm] at (\k,0) {$(0,0,\k)$};
        
        %\draw[blue, thick] (0,0) circle  (10pt);
        %\draw[blue, thick] ({0.5*\k}, {\k*0.866}) circle  (10pt);
        %\draw[blue, thick] (\k,0) circle (10pt);
        \draw[blue, thick] ({0.5*(\k-\alpha_val)}, {(\k-\alpha_val)*0.866}) circle  (10pt);
        %\draw[blue, thick] ({0.5*\alpha_val}, {\alpha_val*0.866}) circle  (10pt);
        %\draw[blue, thick] ({\k-\alpha_val}, 0) circle  (10pt);
        \draw[blue, thick] ({\alpha_val}, 0) circle (10pt);
        \draw[blue, thick] ({\k-0.5*\alpha_val}, {\alpha_val*0.866}) circle (10pt);
        %\draw[blue, thick] ({0.5*(\k+\alpha_val)}, {(\k-\alpha_val)*0.866}) circle  (10pt);
      \end{tikzpicture}
  \end{minipage}\hfill
  \begin{minipage}{0.48\textwidth} % Adjusted width for two figures
      \centering
      \begin{tikzpicture}[scale=0.82] % Increased scale
        \def\k{9}
        \def\alpha_val{2}
        
        \foreach \y in {0,...,\numexpr\k-1} {
          \foreach \x in {\y,...,\numexpr\k-1} {\fill[blue!25] (\x - 0.5*\y, \y*0.866) -- (\x+1 - 0.5*\y, \y*0.866) -- (\x+0.5 - 0.5*\y, \y*0.866+0.866) -- cycle;}
        }
        
        \foreach \y in {0,...,\k} {
          \draw[thick] (0.5*\y, {0.866*\y}) -- ({\k - 0.5*\y}, {0.866*\y});
          \ifnum \y<\k \foreach \x in {\y,...,\numexpr\k-1} {\draw[thick] ({\x+1 - 0.5*\y}, {0.866*\y}) -- ({\x+1- 0.5*(\y+1)}, {0.866*(\y+1)});} \fi
          \ifnum \y>0 \foreach \x in {\y,...,\k} {\draw[thick] ({\x - 0.5*\y}, {0.866*\y}) -- ({\x - 1 - 0.5*(\y-1)}, {0.866*(\y-1)});} \fi
        }
        
        \foreach \y in {0,...,\k} {
          \foreach \x in {\y,...,\k} {
            \def\isproblematic{0}
            \ifnum\x=7 \ifnum\y=7 \def\isproblematic{1} \fi\fi
            \ifnum\x=2 \ifnum\y=0 \def\isproblematic{1} \fi\fi
            \ifnum\x=9 \ifnum\y=2 \def\isproblematic{1} \fi\fi
            
            \def\isreplacement{0}
            \ifnum\x=8 \ifnum\y=8 \def\isreplacement{1} \fi\fi
            \ifnum\x=6 \ifnum\y=6 \def\isreplacement{1} \fi\fi
            \ifnum\x=3 \ifnum\y=0 \def\isreplacement{1} \fi\fi
            \ifnum\x=1 \ifnum\y=0 \def\isreplacement{1} \fi\fi
            \ifnum\x=9 \ifnum\y=1 \def\isreplacement{1} \fi\fi
            \ifnum\x=9 \ifnum\y=3 \def\isreplacement{1} \fi\fi
            
            \ifnum\isproblematic=1
                \draw[fill=white, thick, draw=black] (\x - \y*0.5, \y*0.866) circle (5pt);
            \else
                \ifnum\isreplacement=1
                    \draw[fill=red, thick] (\x - \y*0.5, \y*0.866) circle (5pt);
                \else
                    \pgfmathtruncatemacro{\check}{mod(-\x + \k + 2*\y - 1, 3)}
                    \ifnum \check = 0 \draw[fill=red, thick] (\x - \y*0.5, \y*0.866) circle (5pt);
                    \else \draw[fill=white, thick, draw=black] (\x - \y*0.5, \y*0.866) circle (5pt); \fi
                \fi
            \fi
          }
        }
        \node[above, yshift=2mm] at ({0.5*\k}, {\k*0.866}) {$(0,\k,0)$};
        \node[below, xshift=1mm, yshift=-2mm] at (0,0) {$(\k,0,0)$};
        \node[below, xshift=-1mm, yshift=-2mm] at (\k,0) {$(0,0,\k)$};
  
        \draw[red, dashed, thick] ({0.5*(\k-\alpha_val)}, {(\k-\alpha_val)*0.866}) circle (10pt);
        \draw[red, dashed, thick] ({\alpha_val}, 0) circle (10pt);
        \draw[red, dashed, thick] ({\k-0.5*\alpha_val}, {\alpha_val*0.866}) circle  (10pt);
        
        \draw[red, solid, thick] ({0.5*8}, {8*0.866}) circle  (10pt);
        \draw[red, solid, thick] ({0.5*6}, {6*0.866}) circle  (10pt);
        \draw[red, solid, thick] (3, 0) circle  (10pt);
        \draw[red, solid, thick] (1, 0) circle  (10pt);
        \draw[red, solid, thick] ({9-0.5*1}, {1*0.866}) circle  (10pt);
        \draw[red, solid, thick] ({9-0.5*3}, {3*0.866}) circle  (10pt);
        
      \end{tikzpicture}
  \end{minipage}
  \caption{An illustration of the construction for $k=9, \alpha=2$.}
  \label{fig:final_two_figures_large_scale}
  \end{figure}
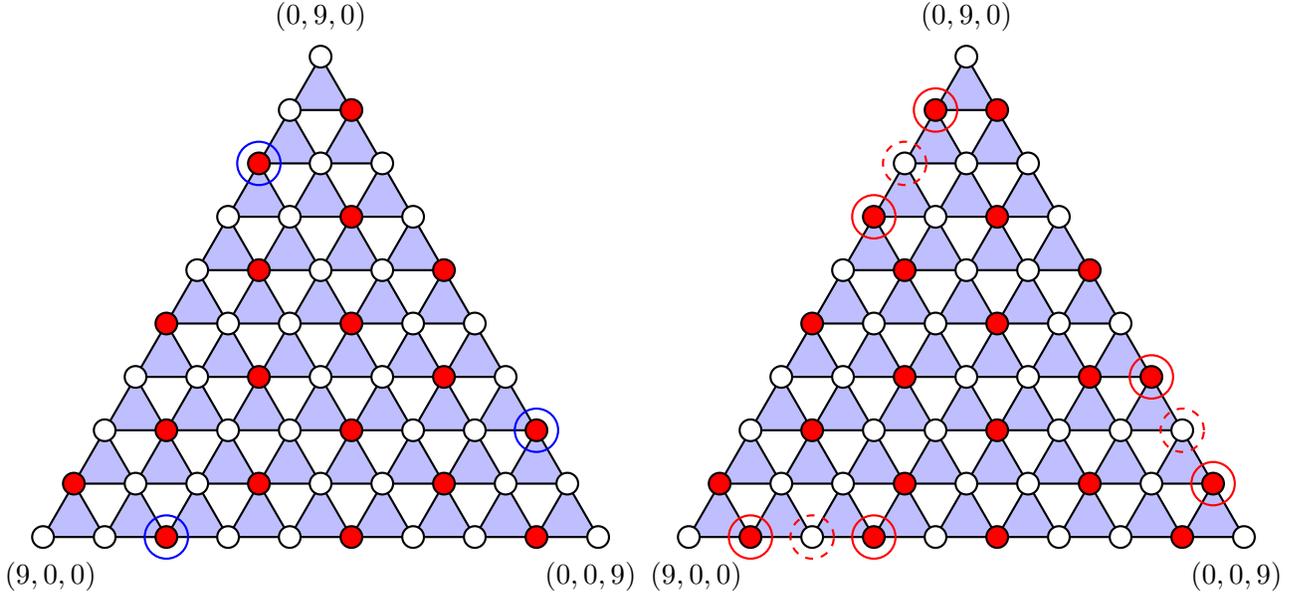

 First, the extension property holds by construction. The proof is analogous to that of \cref{claim:extension_restored} and is therefore omitted. It remains to find a valid invariant $v_C$ for each connected component $C$ of $G_\mathcal{T}$. We begin by proving two useful claims.

\begin{claim}\label{claim:T_disjoint_F}
$\mathcal{T} \cap \mathcal{F} = \emptyset$.
\end{claim}
\begin{proof}
By definition, the first part of the union, $\mathcal{B}_3^k \setminus \mathcal{P}$, is disjoint from $\mathcal{F}$ since $\mathcal{P} = \mathcal{B}_3^k \cap \mathcal{F}$. It is therefore sufficient to show that $(\bigcup_{\vec{x}\in\mathcal{P}} \mathcal{R}_{\vec{x}}) \cap \mathcal{F} = \emptyset$.
Assume for contradiction that there exists a type $\vec{r} \in \mathcal{R}_{\vec{x}}$ for some $\vec{x} \in \mathcal{P}$ such that $\vec{r} \in \mathcal{F}$. Since its parent type $\vec{x}$ is also in $\mathcal{F}$ and they are adjacent, this implies that the forbidden set $\mathcal{F}$ contains two adjacent types. This happens only when $\alpha=1$ or $\alpha=k-1$, or when $k$ is odd and $\alpha=(k-1)/2$ or $\alpha=(k+1)/2$. These conditions on $\alpha$ correspond to $\ell \equiv \pm 1 \pmod k$ or $\ell \equiv \pm 2 \pmod k$, which contradicts the hypothesis on $\ell$.
\end{proof}

\begin{claim}\label{claim:component_one_replacement}
Any two distinct types in $\bigcup_{\vec{x} \in \mathcal{P}} \mathcal{R}_{\vec{x}}$ belong to different connected components of $\mathcal{T}$.
\end{claim}

\begin{proof}

We first prove that for any distinct $\vec{x}_1, \vec{x}_2 \in \mathcal{P}$, we have  $\|\vec{x}_1 - \vec{x}_2\|_1 \ge 10$. 

We prove by contradiction. Assume that  there exists $\vec{x}_1, \vec{x}_2 \in \mathcal{P}$, where $\vec{x}_1 \neq \vec{x}_2$ and $\|\vec{x}_1 - \vec{x}_2\|_1 < 10$. 
Let the difference vector be $\vec{v} = \vec{x}_1 - \vec{x}_2 = (v_1, v_2, v_3)$. Since $\vec{x}_1, \vec{x}_2 \in \mathcal{P} \subseteq \mathcal{B}_3^k$, we have $v_1 + v_2 + v_3 = 0$ and $v_1 + 2v_2 \equiv 0 \pmod 3$. Combining these constraints with the norm condition $\|\vec{v}\|_1 < 10$ reveals that $\vec{v}$ must be a permutation of one of the following vectors: $\pm (2, -1, -1)$, $\pm (4, -2, -2)$, or $\pm (3, 0, -3)$.

Furthermore, recall that $\vec{x}_1, \vec{x}_2 \in \mathcal{F}=\text{Perm}(k,0,0) \cup \text{Perm}(\alpha, k-\alpha, 0)$.  If $\vec{v}= \vec{x}_1 - \vec{x}_2 $ is a permutation of $\pm (2, -1, -1)$ or $\pm (4, -2, -2)$, this implies that either $k=2$ with $\alpha=1$ or $k=4$ with $\alpha=2$. However, such scenarios are impossible because $k \ge 4$ and $\gcd(k, \alpha)=1$. Consequently, $\vec{v}$ must be a permutation of $(3, 0, -3)$. This implies that $\alpha=3$, $k-\alpha=3$, or $|(k-\alpha)-\alpha|=3$, which correspond to $3\ell \equiv \pm 1 \pmod k$ or $3\ell \equiv \pm 2 \pmod k$, a  contradiction to the hypothesis on $\ell$.

Thus for any distinct $\vec{x}_1, \vec{x}_2 \in \mathcal{P}$, we have  $\|\vec{x}_1 - \vec{x}_2\|_1 \ge 10$. As the problematic types are well-separated, the result follows from Claim~\ref{claim:component_separation}.\end{proof}

With the preceding two claims, we now construct a valid invariant $v_C$ for each connected component $C$. We consider two cases.

First, if the connected component $C$ contains no type from $\bigcup_{\vec{x}\in\mathcal{P}} \mathcal{R}_{\vec{x}}$. Then $C \subseteq \mathcal{B}_3^k \setminus \mathcal{P}$. By Claim~\ref{claim:base_family_props}, $C$ must be a singleton $\{\vec{w}\}$. Assume for contradiction that all coordinates of $\vec{w}$ lie in $\{0, k, \alpha, k-\alpha\}$. Since $\vec{w} \in \mathcal{B}_3^k \setminus \mathcal{P}$, it is by definition not in the forbidden family $\mathcal{F}$, so it cannot be a permutation of $(k,0,0)$ or $(\alpha, k-\alpha, 0)$. The only remaining possibility is that its non-zero coordinates are identical (e.g., $(\alpha, \alpha, \alpha)$ when $3\alpha=k$, or $(k-\alpha, k-\alpha, k-\alpha)$ when $3(k-\alpha)=k$), but these are ruled out by the condition $\gcd(k,\alpha)=1$ with $k \ge 4$. Thus, at least one coordinate $w_j$ must exist outside this set. We set $I_C=\{j\}$ and $v_C=w_j$, which provides a valid invariant.

Next, if the connected component $C$ contains some type from $\bigcup_{\vec{x}\in\mathcal{P}} \mathcal{R}_{\vec{x}}$. By \cref{claim:component_one_replacement}, it contains exactly one type from $\bigcup_{\vec{x}\in\mathcal{P}} \mathcal{R}_{\vec{x}}$. By symmetry, it suffices to analyze the connected components generated by types in $\mathcal{R}_{\vec{x}}$ for $\vec{x} \in \{(k,0,0), (\alpha, k-\alpha, 0)\}$.
\begin{itemize}
    \item If $\vec{x}=(k,0,0)$, then $\mathcal{R}_{\vec{x}} = \{(k-1,1,0)\}$. The connected component is $C = \{(k-1,1,0), (k-2,1,1)\}$. We set $I_C=\{2\}$ and $v_C=1$.
    \item If $\vec{x}=(\alpha, k-\alpha, 0)$, then $\mathcal{R}_{\vec{x}} = \{(\alpha-1, k-\alpha+1, 0), (\alpha+1, k-\alpha-1, 0)\}$. These two types form two distinct connected components.
    
    For the connected component $C_1 = \{(\alpha-1, k-\alpha+1, 0), (\alpha-2, k-\alpha+1, 1)\}$, we set $I_{C_1}=\{2\}$ and $v_{C_1}=k-\alpha+1$. As $\ell \not \equiv 2 \pmod k$, we have $2\alpha \not \equiv 1 \pmod k$, therefore $v_{C_1}=k-\alpha+1\not \equiv \alpha \pmod k$.

    For the connected component $C_2 = \{(\alpha+1, k-\alpha-1, 0), (\alpha+1, k-\alpha-2, 1)\}$, we set $I_{C_2}=\{1\}$ and $v_{C_2}=\alpha+1$. As $\ell \not \equiv -2 \pmod k$, we have $2\alpha \not \equiv  -1 \pmod k$, therefore $v_{C_2}=\alpha+1 \not \equiv k-\alpha \pmod k$.
\end{itemize}
It is straightforward to verify that $\sum_{i \in I_C} w_i = v_C$ for any $(w_1,w_2,w_3) \in C$. Since  $\ell \not \equiv \pm 1 \pmod k$, we know that $\alpha \not \equiv \pm 1 \pmod k$. Therefore, in all cases $v_C\notin\{0,\alpha,k-\alpha,k\}$. This implies that in all cases $\frac{k}{\gcd(k,\ell)}=k \nmid v_C$ and $\ell v_C\not \equiv \pm 1 \pmod k$.  This completes the proof.
\end{proof}

\end{document}